\documentclass[11pt,a4paper]{article}
\usepackage{amsmath}
\usepackage{amssymb}
\usepackage{amsthm}
\usepackage[english]{babel}
\usepackage{verbatim}
\usepackage[shortlabels]{enumitem}
\usepackage{tikz-cd}
\usepackage{etoolbox} 
\usepackage[all]{xy}
\usepackage{hyperref}

\newtheoremstyle{mio}%
	{}{} 
	{\itshape}{} 
	{\bfseries}{.}{ } 
	{#1 #2\thmnote{~\mdseries(#3)}} 
\theoremstyle{mio}
\newtheorem{teor}{Theorem}[section]
\newtheorem{cor}[teor]{Corollary}
\newtheorem{prop}[teor]{Proposition}
\newtheorem{lemma}[teor]{Lemma}

\newtheorem{Prop}[teor]{Proposition}
\newtheorem{Thm}[teor]{Theorem}
\newtheorem{Cor}[teor]{Corollary}
\theoremstyle{definition}
\newtheorem{Def}[teor]{Definition}
\newtheorem{Ex}[teor]{Example}
\newtheorem{Rem}[teor]{Remark}
\newtheorem{oss}[teor]{Remark}

\newcommand{\Br}{\mathrm{Br}}

\DeclareMathOperator{\degdom}{degdom}

\newcommand{\ins}[1]{\mathbb{#1}}
\newcommand{\insN}{\ins{N}}
\newcommand{\insZ}{\ins{Z}}
\newcommand{\insQ}{\ins{Q}}
\newcommand{\insR}{\ins{R}}
\newcommand{\inN}{\in\insN}
\newcommand{\inZ}{\in\insZ}
\newcommand{\inR}{\in\insR}
\newcommand{\N}{\mathbb{N}}
\newcommand{\Z}{\mathbb{Z}}
\newcommand{\Zar}{\mathrm{Zar}}
\newcommand{\inv}[1]{\frac{1}{#1}}

\newcommand{\insiemeVE}{\mathcal{V}}
\newcommand{\insiemeVEalg}{\mathcal{V}_\mathrm{alg}}

\newcommand{\cons}{\mathrm{cons}}

\newcommand{\limiti}{\mathcal{L}}
\newcommand{\cball}{\mathrm{CBall}}

\newcommand{\corona}{\mathcal{C}}

\newcommand{\Int}{\mathrm{Int}}
\newcommand{\vEXbeta}{\Delta}

\newcommand{\R}{\mathbb{R}}

\newcommand{\caso}[1]{\medskip\underline{\textbf{Case #1.}}~\smallskip}

\title{The Zariski-Riemann space of valuation domains associated to pseudo-convergent sequences}

\author{
G.\ Peruginelli\footnote{Dipartimento di Matematica "Tullio Levi-Civita", University of Padova, Via Trieste 63,
35121 Padova, Italy. E-mail: gperugin@math.unipd.it}\\
D. Spirito\footnote{Dipartimento di Matematica e Fisica, University of Roma Tre, Largo San Leonardo Murialdo 1, 00146 Roma. E-mail: spirito@mat.uniroma3.it; \emph{Current address}: Dipartimento di Matematica "Tullio Levi-Civita", University of Padova, Via Trieste 63,
35121 Padova, Italy. E-mail: spirito@math.unipd.it}
}

\begin{document}
\noindent\leftmark{Trans. Amer. Math. Soc. 373 (2020), no. 11, 7959-7990.\\ \footnotesize{\href{https://arxiv.org/abs/1809.09539}{https://arxiv.org/abs/1809.09539}},\;\; \href{http://dx.doi.org/10.1090/tran/8185}{http://dx.doi.org/10.1090/tran/8185}}
{\let\newpage\relax\maketitle} 

\begin{abstract}
\noindent Let $V$ be a valuation domain with quotient field $K$. Given a pseudo-convergent sequence $E$ in $K$, we study two constructions associating to $E$ a valuation domain of $K(X)$ lying over $V$, especially when $V$ has rank one. The first one has been introduced by Ostrowski, the second one more recently by Loper and Werner. We describe the main properties of these valuation domains, and we give a notion of equivalence on the set of pseudo-convergent sequences of $K$  characterizing when the associated valuation domains are equal. Then, we analyze the topological properties of the Zariski-Riemann spaces formed by these valuation domains. \\


\noindent MSC Primary 12J20, 13A18, 13F30.
\end{abstract}

\section{Introduction}

Let $V$ be a valuation domain with quotient field $K$. Determining and describing all the extensions of $V$ to the field $K(X)$ of rational functions is an old and well-studied problem, which plays a vital role in several topics in field theory, commutative algebra and beyond (see for example \cite{Kuh} and the references therein). The problem has been approached in a few different ways, with the main ones being through \emph{key polynomials} (starting from the work of MacLane \cite{maclane} and developed, among many others, by Vaqui\'e \cite{vaquie}), \emph{minimal pairs} (introduced by Alexandru, Popescu and Zaharescu \cite{APZ1,APZ2}) and \emph{pseudo-convergent sequences}. The latter were introduced by Ostrowski in \cite{Ostr}, who used them to describe all rank one extensions of the rank one valuation domain $V$ to $K(X)$; subsequently, Kaplansky  used this notion in \cite{Kap} for valuation domains of any rank to characterize immediate extensions of valuation domains and maximally valued fields. More recently, Chabert in \cite{ChabPolCloVal} generalized  Ostrowski's definition by means of \emph{pseudo-monotone sequences}  to characterize the polynomial closure of subsets of valued fields of rank one, an important topic in the study of integer-valued polynomials.

In this paper,  we study two constructions of extensions of $V$ to $K(X)$ associated to a pseudo-convergent sequence $E\subset K$. The first one, which we denote by $V_E$, is the same construction introduced by Loper and Werner \cite{LW} for certain kinds of pseudo-convergent sequences on rank one valuation domains; we show that it actually defines a valuation domain for every pseudo-convergent sequence and for valuation domains of any rank (Theorem \ref{Thm:VE valuation domain}). The second one, which we denote by $W_E$, applies only when $V$ has rank one and $E$ satisfies some conditions, and is defined through its valuation $w_E$, which was already introduced by Ostrowski \cite{Ostr}; for this reason, we name it the \emph{Ostrowski valuation} associated to $E$. In Sections \ref{sect:VE} and \ref{sect:WE}, we investigate the structure of these valuation domains, in particular when $V$ has rank one; among other things, we show that $V_E\subseteq W_E$, we characterize when $V_E$ has rank 2 (Theorem \ref{teor:rank-VE}), we find their value group and residue field, and we describe explicitly the valuation $v_E$ associated to $V_E$ as a map from $K(X)$ to $\insR^2$ (Theorem \ref{teor:valut-dim2}). Many of these results are based on a general theorem (Theorem \ref{teor:linear}) which expresses the valuation of $\phi(t)$, for a rational function $\phi(X)$, as a linear function of $v(t-s)$, in an annulus of center $s$ which contains neither poles nor zeros of $\phi(X)$.

In Section \ref{sect:equiv}, following Ostrowski, we investigate the notion of  equivalence between two pseudo-convergent sequences, analogous to the concept of equivalence between two Cauchy sequences. We show that two pseudo-convergent sequences are equivalent if and only if their associated valuation rings are equal;  moreover, if they are of algebraic type then these conditions are also equivalent to the property of having the same set of pseudo-limits (in the algebraic closure of $K$ and with respect to the same extension of $v$; see Theorem \ref{teor:equiv-VE}). We also give a geometric interpretation of this fact in Section \ref{sect:geomint}. Using these results, we show that the extensions of an Ostrowski valuation $w_E$ to $\overline{K}(X)$ is completely determined by its restriction to $\overline{K}$ (Theorem \ref{Teor:extension wE}).

In Section \ref{sect:spaces}  we study the spaces $\insiemeVE$ and $\mathcal{W}$ formed, respectively, by the rings of the form $V_E$ and by the rings of the form $W_E$, from a topological point of view; more precisely, we study the Zariski and the constructible topologies they inherit from the Zariski-Riemann space $\Zar(K(X)|V)$. In particular, we first analyze the difference between these two topologies, showing that they coincide on $\mathcal{W}$ (Proposition \ref{prop:spazioW}), while they coincide on $\insiemeVE$ if and only if the residue field of $V$ is finite (Proposition \ref{prop:costr-VE}); we also show how $V_E$ can be seen as a limit of valuation domains defined from the members of $E$, mirroring the fact that (classes of) Cauchy sequences can be associated to their limit points (Proposition \ref{prop:limitcostr}). In Section \ref{sect:separation}, we show that $\insiemeVE$, endowed with the Zariski topology, is a regular space and we deduce a sufficient condition for $\insiemeVE$ to be metrizable.


\section{Background and notation}\label{sect:background}
Throughout the article, $V$ is a valuation domain; we denote by $K$ its quotient field, by $M$ its maximal ideal and by $v$ the valuation associated to $V$. Its value group is denoted by $\Gamma_v$. We denote by $\widehat{K}$ and $\widehat{V}$ the completion of $K$ and $V$, respectively, with respect to the topology induced by the valuation $v$. We still denote by $v$ the unique extension of $v$ to $\widehat K$ (whose valuation domain is precisely $\widehat{V}$). We denote by $\overline{K}$ a fixed algebraic closure of $K$. If $u$ is an extension of $v$ to $\overline{K}$, then the value group of $u$ is the divisible hull of $\Gamma_v$, i.e., $\insQ\Gamma_v=\insQ\otimes_\insZ\Gamma_v$. 

The basic objects of study of this paper are pseudo-convergent sequences, introduced by Ostrowski in \cite{Ostr} and used by Kaplansky in \cite{Kap} to describe immediate extensions of valued fields. Related concepts are  \emph{pseudo-stationary} and \emph{pseudo-divergent} sequences \cite{ChabPolCloVal}, which we consider in \cite{PS2}.
\begin{Def}
Let $E=\{s_n\}_{n\in\N}$ be a sequence in $K$. We say that $E$ is a \emph{pseudo-convergent} sequence if $v(s_{n+1}-s_n)<v(s_{n+2}-s_{n+1})$ for all $n\in\N$.
\end{Def}
In particular, if $E=\{s_n\}_{n\in\N}$ is a pseudo-convergent sequence and $n\geq 1$, then $v(s_{n+k}-s_n)=v(s_{n+1}-s_n)$ for all $k\geq 1$. We shall usually denote this quantity by $\delta_n$; following \cite[p. 327]{Warner} we call the sequence $\{\delta_n\}_{n\inN}$ the \emph{gauge} of $E$.

We shall make use of the following notation: given a sequence of real numbers   $\{r_n\}_{n\in\N}$ and $r\in\R\cup\{\infty\}$, we write $r_n\nearrow r$ if $r_n$ is a strictly increasing sequence with limit $r$. 

\begin{Def}
The \emph{breadth ideal} of $E$ is
\begin{equation*}
\Br(E)=\{b\in K \mid v(b)>v(s_{n+1}-s_n),\forall n\in\N\}.
\end{equation*}
\end{Def}
In general, $\Br(E)$ is a fractional ideal of $V$ and may not be contained in $V$.

The following definition has been introduced in \cite{Kap}, even though already in \cite[p. 375]{Ostr} an equivalent concept appears (see \cite[X, p. 381]{Ostr} for the equivalence).
\begin{Def}
An element $\alpha\in K$ is a \emph{pseudo-limit} of $E$ if $v(\alpha-s_n)<v(\alpha-s_{n+1})$ for all $n\in\N$, or, equivalently, if $v(\alpha-s_n)=\delta_n$ for all $n\in\N$. We denote the set of pseudo-limits of $E$ by $\limiti_E$, or $\limiti_E^v$ if we need to emphasize the valuation.
\end{Def}
If $\Br(E)$ is the zero ideal then $E$  is a Cauchy sequence in $K$ and converges to an element of $\widehat{K}$, which is the unique pseudo-limit of $E$. In general, Kaplansky proved the following more general result.
\begin{lemma}\label{lemma:kaplansky}
\cite[Lemma 3]{Kap} Let $E\subset K$ be a pseudo-convergent sequence. If $\alpha\in K$ is a pseudo-limit of $E$, then the set of pseudo-limits of $E$ in $K$ is equal to $\alpha+\Br(E)$. 
\end{lemma}

If $w$ is an extension of $v$ to a field $L$ containing $K$, and $E$ is a sequence in $K$, then $E$ is pseudo-convergent with respect to $w$ if and only if $E$ is pseudo-convergent with respect to $v$. Moreover, every pseudo-limit of $E$ under $v$ in $K$ is also a pseudo-limit under $w$. 

Suppose now that $V$ has rank one; then we consider $\Gamma_v$ and $\insQ\Gamma_v$ as totally ordered subgroups of $\insR$. The valuation $v$ induces an ultrametric distance $d$ on $K$, defined by
\begin{equation*}
d(x,y)=e^{-v(x-y)}.
\end{equation*}
In this metric, $V$ is the closed ball of center 0 and radius 1. Given $s\in K$ and $\gamma\in\Gamma_v$, we denote the ball of center $s$ and radius $r=e^{-\gamma}$ by:
$$B(s,r)=\{x\in K \mid d(x,s)\leq r\}=\{x\in K \mid v(x-y)\geq \gamma\}.$$
A ball in $\overline{K}$ with respect to an extension $u$ of $v$ is denoted by $B_u(s,r)$.

If $E\subset K$ is a pseudo-convergent sequence, then the gauge $\{\delta_n\}_{n\in\N}$ of $E$ is a strictly increasing sequence of real numbers, and so the following definition makes sense.
\begin{Def}
The \emph{breadth} of a pseudo-convergent sequence $E=\{s_n\}_{n\in\N}$ is the limit
\begin{equation*}
\delta_E=\lim_{n\to\infty}v(s_{n+1}-s_n)=\lim_{n\to\infty}\delta_n.
\end{equation*}
\end{Def}
The breadth $\delta$ is an element of $\insR\cup\{\infty\}$, and it may not lie in $\Gamma_v$. We can use the breadth to characterize the breadth ideal: indeed, $\Br(E)=\{b\in K\mid v(b)\geq\delta_E\}$, or equivalently $\delta_E=\inf\{v(b)\mid b\in\Br(E)\}$. If $\delta=+\infty$, then $\Br(E)$ is just the zero ideal and $E$ is a Cauchy sequence in $K$. If $V$ is a discrete valuation ring, then every pseudo-convergent sequence is actually a Cauchy sequence. Lemma \ref{lemma:kaplansky} can also be phrased in a geometric way: if $\alpha\in\limiti_E$, then $\limiti_E$ is the closed ball of center $\alpha$ and radius $e^{-\delta_E}$, i.e., $\limiti_E=B(\alpha,e^{-\delta_E})$.

The following concepts have been given by Kaplansky in \cite{Kap} in order to study the different kinds of immediate extensions of a valued field $K$. Recall that  if $L$ is a field extension of $K$, a  valuation domain $W$ of $L$ \emph{lies over} $V$ if $W\cap K=V$. In this case, the residue field of $W$ is naturally an extension of the residue field of $V$ and similarly the value group of $W$ is an extension of the value group of $V$. We say that $W$ is immediate over $V$ if both the residue fields and the value groups are the same.
\begin{Def}
Let $E$ be a pseudo-convergent sequence. We say that $E$ is of \emph{transcendental type} if $v(f(s_n))$ eventually stabilizes for every $f\in K[X]$; on the other hand, if $v(f(s_n))$ is eventually increasing for some $f\in K[X]$, we say that $E$ is of \emph{algebraic type}.
\end{Def}
The main difference between these two kind of sequences is the nature of the pseudo-limits: if $E$ is of algebraic type, then $E$ has pseudo-limits in $\overline{K}$ (for some extension $u$ of $v$), while if $E$ is of transcendental type then $E$ admits a pseudo-limit only in a transcendental extension \cite[Theorems 2 and 3]{Kap}. 

If $L=K(X)$ and $W$ lies over $V$, then $W$ is said to be a \emph{residually transcendental extension} of $V$ (or simply residually transcendental if $V$ is understood) if the residue field of $W$ is a transcendental extension of the residue field of $V$ \cite{APZ1}.

\begin{Def}
Let $\Gamma$ be a totally ordered group containing $\Gamma_v$, and take $\alpha\in K$ and $\delta\in\Gamma$. The \emph{monomial valuation} $v_{\alpha,\delta}$ is defined in the following way: if $f(X)\in K[X]$ is a polynomial, write $f(X)=a_0+a_1(X-\alpha)+\ldots+a_n(X-\alpha)^n$; then,
\begin{equation*}
v_{\alpha,\delta}(f)=\inf\{v(a_i)+i\delta \mid i=0,\ldots,n\}.
\end{equation*}
\end{Def}
It is well known that $v_{\alpha,\delta}$ naturally extends to a valuation on $K(X)$ \cite[Chapt. VI, \S. 10, Lemme 1]{Bourb}, and $v_{\alpha,\delta}$ is residually transcendental over $v$ if and only if  $\delta$ has finite order over $\Gamma_v$ \cite[Lemma 3.5]{PerPrufer}. Furthermore, every residually transcendental extension of $V$ can be written as $W\cap K(X)$, where $W$ is a valuation domain of $\overline{K}(X)$ associated to a monomial valuation \cite{AP,APZ1}.

%
Let $D$ be an integral domain and $L$ be a field containing $D$ (not necessarily the quotient field of $D$). The \emph{Zariski space of $D$ in $L$}, denoted by $\Zar(L|D)$, is the set of valuation domains of $L$ containing $D$ endowed with the so-called \emph{Zariski topology}, i.e., with the topology generated by the subbasic open sets
\begin{equation*}
B(\phi)=\{V\in\Zar(L|D)\mid \phi\in V\},
\end{equation*}
where $\phi\in L$. Under this topology, $\Zar(L|D)$ is a compact space \cite[Chapter VI, Theorem 40]{ZS2} that is almost never Hausdorff nor $T_1$ (indeed, $\Zar(L|D)$ is a $T_1$ space if and only if $D$ is a field and $L$ is an algebraic extension of $D$).

The \emph{constructible topology} (also called \emph{patch topology}) on $\Zar(L|D)$ is the coarsest topology such that the subsets $B(\phi_1,\ldots,\phi_k)$ are both open and closed; we denote this space by $\Zar(L|D)^\cons$. Clearly, the constructible topology is finer than the Zariski topology; however, $\Zar(L|D)^\cons$ is still compact, and furthermore it is always Hausdorff \cite[Theorem 1]{hochster_spectral}.

\section{A valuation domain associated to a pseudo-convergent sequence}\label{sect:VE}

The following valuation domain associated to a pseudo-convergent sequence has been introduced by Loper and Werner in \cite{LW} in the case of a valuation domain $V$ of $K$ of rank one. We generalize their construction to valuation domains of arbitrary rank.
\begin{Def}
Let $E=\{s_n\}_{n\in\N}\subset K$ be a pseudo-convergent sequence. Let
\begin{equation}\label{V_E}
V_{E}=\{\phi\in K(X) \mid \phi(s_n)\in V,\textnormal{ for all but finitely many }n\in\N\}.
\end{equation}
\end{Def}
The aim of this section is to prove that $V_E$ is a valuation domain of $K(X)$ for every pseudo-convergent sequence $E$. When the rank of $V$ is one and $E$ is of transcendental type or has zero breadth ideal, this result was already obtained, respectively, in  Proposition 5.5 and Theorem 5.8 of \cite{LW}. More generally, for any valuation domain, when $E$ is of transcendental type $V_E$ coincides with the valuation domain of $K(X)$ defined by Kaplansky in \cite[Theorem 2]{Kap}, which is an immediate extension of $V$. Since in this case for each $\phi\in K(X)$ we have that $v(\phi(s_n))$ is eventually constant, the value of $\phi$ with respect to the above valuation is equal to that constant value that $\phi(X)$ assumes over $E$. Following this example, our method is heavily based on understanding the values of $\phi(X)$ along a pseudo-convergent sequence.

For the next result, which is not a priori related to pseudo-convergent sequences,  we introduce some notations and definitions.

\begin{Def}\label{weighted sum}
Let $\phi\in K(X)$.  The \emph{multiset of critical points} of $\phi(X)$ is the multiset $\Omega_{\phi}$ of zeroes and poles of $\phi$ in $\overline{K}$ (each of them counted with multiplicity). 
Given a sub-multiset $S=\{\alpha_1,\ldots,\alpha_k\}$ of $\Omega_{\phi}$, by the \emph{weighted sum} of $S$ we mean the sum $\sum_{\alpha_i\in S}\epsilon_i$, where $\epsilon_i$ is equal either to $1$ or to $-1$, according to whether $\alpha_i$ is a zero or a pole of $\phi$, respectively. By the $S$-\emph{part} of $\phi$ we mean the rational function $\phi_S(X)=\prod_{\alpha_i\in S}(X-\alpha_i)^{\epsilon_i}$, where $\epsilon_i\in\{\pm1\}$ is as above. Note that $\phi_{\Omega_{\phi}}(X)$ is equal to $\phi(X)$ up to a constant. 
\end{Def}

Given a convex subset $\Delta$ of $\Gamma_v$, $\beta\in\overline{K}$ and an extension $u$ of $v$ to $\overline K$, we set 
\begin{equation}\label{CbetaDelta}
\corona_u(\beta,\Delta)=\{s\in\overline{K}\mid u(s-\beta)\in\Delta\}
\end{equation}
and, if $\gamma\in\insQ\Gamma_v$, we write $\gamma<\Delta$ ($\gamma>\Delta$, respectively) if $\gamma<\delta$ ($\gamma>\delta$, respectively) for every $\delta\in\Delta$.
\begin{teor}\label{teor:linear}
Let $\phi\in K(X)$ and let $s\in K$; let $u$ be an extension of $v$ to $\overline{K}$. Let $\Delta$ be a convex subset of  $\insQ\Gamma_v$ such that $C=\corona_u(s,\Delta)$ does not contain any critical point of $\phi$. Let $\lambda\inZ$ be equal to the weighted sum of the multiset $S$ of critical points $\alpha$ of $\phi$ which satisfy $u(\alpha-s)>\Delta$ and let $\gamma=u\left(\frac{\phi}{\phi_S}(s)\right)$. Then, for all $t\in C\cap K$, we have
\begin{equation}\label{eq:vphit-lineare}
v(\phi(t))=\lambda v(t-s)+\gamma.
\end{equation}
\end{teor}
\begin{proof}
Over $\overline{K}$, we can write $\phi(X)$ as a product $c\prod_{i=1}^n(X-\alpha_i)^{\epsilon_i}$, where the $\alpha_i$ are the critical points of $\phi$, $\epsilon_i\in\{-1,1\}$ and $c\in K$. Let $C=\corona_u(s,\Delta)\subset\overline K$ and let $t\in K\cap C$. If $u(\alpha_i-s)<\Delta$ then $u(t-\alpha_i)=u(s-\alpha_i)$, while if $u(\alpha_i-s)> \Delta$ then $u(t-\alpha_i)=u(s-t)$ (note that, by assumption, there are no other possibilities for the critical points of $\phi(X)$). Therefore, we have
\begin{align*}
v(\phi(t)) & =v(c)+\sum_{i:u(\alpha_i-s)<\Delta}\epsilon_iu(t-\alpha_i)+\sum_{i:u(\alpha_i-s)> \Delta}\epsilon_iu(t-\alpha_i)=\\
& =v(c)+\sum_{i:u(\alpha_i-s)<\Delta}\epsilon_i u(\alpha_i-s)+\sum_{i:u(\alpha_i-s)> \Delta}\epsilon_i u(t-s)=\gamma+\lambda v(t-s)
\end{align*}
where $\lambda=\sum_{i: u(\alpha_i-s)>\Delta}\epsilon_i$ and $\gamma=v(c)+\sum_{i:u(\alpha_i-s)<\Delta}\epsilon_i u(s-\alpha_i)=u\left(\frac{\phi}{\phi_S}(s)\right)$, with $S$ being the multiset of critical points $\alpha_i$ of $\phi$ which satisfy $u(\alpha_i-s)>\Delta$. In particular, $\lambda\in\insZ$ and $\gamma\in\insQ\Gamma_v$ do not depend on $t$. The claim is proved.
\end{proof}

\begin{oss}\label{oss:linear}
Let $\alpha_1,\ldots,\alpha_n$ be the zeros and the poles of $\phi$, and let $\rho_i=u(s-\alpha_i)$; without loss of generality, suppose $\rho_1<\cdots<\rho_n$. Then, the sets $\Delta_i=(\rho_i,\rho_{i+1})$, for $i=0,\ldots,n$ (with the convention $\rho_0=-\infty$ and $\rho_{n+1}=+\infty$) are the maximal convex sets on which Theorem \ref{teor:linear} can be applied: that is, they satisfy (by definition) the hypothesis of the theorem, and if $\Delta_i\subsetneq \Delta$ for some other convex set $\Delta$ then the theorem cannot be applied to $\Delta$.
\end{oss}

In order to apply Theorem \ref{teor:linear} to pseudo-convergent sequences, we need the following definition.

\begin{Def}
Let $E=\{s_n\}_{n\inN}$ be a pseudo-convergent sequence in $K$, let $u$ be an extension of $v$ to $\overline{K}$ and let $\phi\in K(X)$. The \emph{dominating degree} $\degdom_{E,u}(\phi)$ of $\phi$ with respect to $E$ and $u$ is the weighted sum of the critical points of $\phi(X)$ (according to Definition \ref{weighted sum}) which are pseudo-limits of $E$ with respect to $u$.
\end{Def}
Note that, by definition, if $E$ is a pseudo-convergent sequence of transcendental type, then $\degdom_{E,u}(\phi)=0$ for every $\phi\in K(X)$.

The following result shows that the values of a rational function over a pseudo-convergent sequence $E\subset K$ form a sequence that is eventually monotone, either strictly increasing, strictly decreasing or stationary, according to whether the dominating degree of $\phi$ with respect to $E$ is positive, negative or equal to zero,  respectively.
\begin{Prop}\label{prop:qsn}
Let $E=\{s_n\}_{n\inN}\subset K$ be a pseudo-convergent sequence with gauge $\{\delta_n\}_{n\inN}$, and let $u$ be an extension of $v$ to $\overline{K}$. Let $\phi\in K(X)$.
\begin{enumerate}[(a)]
\item\label{prop:qsn:qsn} If $\lambda=\degdom_{E,u}\phi$, then there is $\gamma\in\Gamma_v$ such that, for all sufficiently large $n$, we have
\begin{equation*}
v(\phi(s_n))=\lambda\delta_n+\gamma.
\end{equation*}
\item\label{prop:qsn:alg} If $\beta\in\overline K$ is a pseudo-limit of $E$ with respect to $u$, then $\gamma=u\left(\frac{\phi}{\phi_S}(\beta)\right)$, where $S$ is the set of critical points of $\phi(X)$ which are pseudo-limits of $E$.
\item\label{prop:qsn:ext} The dominating degree of $\phi$ does not depend on $u$; that is, if $u'$ is another extension of $v$ to $\overline{K}$, then $\degdom_{E,u}\phi=\degdom_{E,u'}\phi$.
\end{enumerate}
\end{Prop}
\begin{proof}
If $E$ is of transcendental type, then $\lambda=0$ and all claims follow from the definition.

Suppose that $E$ is of algebraic type and $\beta\in\mathcal L_E^u$; we will prove \ref{prop:qsn:qsn} and \ref{prop:qsn:alg} together. Let $\Delta=\Delta_E$ be the least initial segment of $\insQ\Gamma_v$ containing the gauge of $E$. There exists $\tau\in\Gamma_v\cap\Delta$ such that $C=\mathcal C_u(\beta,\Delta\cap(\tau,+\infty))$ contains no critical points of $\phi$. Let $\lambda$ be the weighted sum of the subset $S$ of $\Omega_\phi$ of those elements $\alpha$ such that $u(\alpha-\beta)>\Delta\cap(\tau,+\infty)$ (or, equivalently, $u(\alpha-\beta)>\Delta$) and $\gamma=u\left(\frac{\phi}{\phi_S}(\beta)\right)$. For all $n$ sufficiently large $s_n\in C$: by Theorem \ref{teor:linear}, it follows that for each such $n$ we have
\begin{equation*}
v(\phi(s_n))=\lambda u(\beta-s_n)+\gamma=\lambda\delta_n+\gamma.
\end{equation*}
Note that $\gamma\in\Gamma_v$ and, by Lemma \ref{lemma:kaplansky}, $S$ is the set of critical points of $\phi(X)$ which are pseudo-limits of $E$, so $\lambda$  is the dominating degree of $\phi$ with respect to $E$.

For \ref{prop:qsn:ext}, we note that $v(\phi(s_n))$ does not depend on the extension $u$; hence, if $\lambda=\degdom_{E,u}\phi$, $\lambda'=\degdom_{E,u'}\phi$, $\gamma=u\left(\frac{\phi}{\phi_S}(\beta)\right)$, $\gamma'=u'\left(\frac{\phi}{\phi_S}(\beta)\right)$, we have
\begin{equation*}
v(\phi(s_n))=\lambda\delta_n+\gamma=\lambda'\delta_n+\gamma'
\end{equation*}
for all large $n$. However, this clearly implies $\lambda=\lambda'$, as claimed.
\end{proof}

In view of point \ref{prop:qsn:ext} of the previous proposition, we denote the dominating degree of $\phi$ with respect to $E$ simply as $\degdom_E\phi$.

The term \emph{dominating degree} comes from the following property. We remark that, in a different context, a similar argument has been given in the proof of \cite[Proposition 4.8]{ChabPolCloVal}.
\begin{prop}\label{prop:degdom}
Let $E=\{s_n\}_{n\inN}\subset K$ be a pseudo-convergent sequence with pseudo-limit $\beta\in K$, and let $f(X)=\sum_{i=0,\ldots,d} a_i(X-\beta)^i\in K[X]$. Then, $\degdom_Ef$ is the non-negative integer $k$ such that $v(f(s_n))=v(a_k(s_n-\beta)^k)$ for all large $n$.
\end{prop}
\begin{proof}
Clearly, if $v(f(s_n))=v(a_k(s_n-\beta)^k)$ for all large $n$ then $v(f(s_n))=k\delta_n+v(a_k)$ and so $k=\degdom_Ef$.

Conversely, suppose $k=\degdom_Ef$. Then, by definition, $v(f(s_n))=k\delta_n+\gamma$ for some $\gamma\in\Gamma_v$ (for all large $n$), where $\{\delta_n\}_{n\inN}$ is the gauge of $E$.  We consider the following linear functions from $\Gamma_v$ to $\Gamma_v$:
\begin{equation*}
\lambda_i(\eta)=i\eta+v(a_i),\;i\in\{0,\ldots,d\}.
\end{equation*}
Let $\Delta$ be the least initial segment of $\Gamma_v$ containing the gauge of $E$: then, since the $\lambda_i$ are linear, there is a $\tau\in\Delta$ and an $r\in\{0,\ldots,d\}$ such that $\lambda_r(\eta)<\lambda_i(\eta)$ for all $\eta\in\Delta\cap(\tau,+\infty)$ and for all $i\not=r$. Therefore, whenever $\delta_n\in\Delta\cap(\tau,+\infty)$ we must have
\begin{equation*}
v(f(s_n))=v\left(\sum_ia_i(s_n-\beta)^i\right)=\inf_i\{v(a_i(s_n-\beta)^i)\}= r\delta_n+v(a_r).
\end{equation*}
In particular, it must be $r=k$, as claimed.
\end{proof}
%

\begin{Thm}\label{Thm:VE valuation domain}
Let $E=\{s_n\}_{n\in\N}$ be a pseudo-convergent sequence. Then $V_E\subset K(X)$ is a valuation domain lying over $V$ with maximal ideal equal to $M_{V_E}=\{\phi\in K(X)\mid v(\phi(s_n))\in M,\textnormal{ for all but finitely many }n\in\N\}$. Moreover, $X$ is a pseudo-limit of $E$ with respect to the valuation $v_E$ associated to $V_E$.
\end{Thm}
\begin{proof}
Clearly, $V_E$ is a ring and $V_E\cap K=V$.

If $E$ is of transcendental type then $V_E$ is exactly the valuation domain of the immediate extension of the valuation $v$ to $K(X)$ induced by $E$ as in \cite[Theorem 2]{Kap}. We have that $X$ is a pseudo-limit of $E$ by \cite[Theorem 2]{Kap}.

Suppose now that $E$ is of algebraic type, and let $\phi\in K(X)$. By Proposition \ref{prop:qsn}, $v(\phi(s_n))$ is a linear function of $\delta_n$; hence, it is either eventually positive, eventually zero or eventually negative. Since $v(\phi^{-1}(s_n))=-v(\phi(s_n))$ (provided that $\phi(s_n)\neq 0$, which happens only finitely many times), we have that $\phi\in V_E$, in the first and second case, while in the third case $\phi^{-1}\in V_E$. Hence, $V_E$ is a valuation domain, and the claim about the maximal ideal follows easily.



Finally, we show that $X$ is a pseudo-limit of $E$ with respect to  $v_E$. Fix $n\inN$, and let $\phi(X)=\frac{X-s_{n+1}}{X-s_n}$. Then, for $m>n+1$ we have $v(\phi(s_m))=\delta_{n+1}-\delta_n>0$, and thus $v_E(X-s_{n+1})>v_E(X-s_n)$. It follows that $X$ is a pseudo-limit of $E$, as claimed.
\end{proof}

We want now to link the valuation domain $V_E$ to another class of valuation domains, which for example have been recently considered in \cite{PerTransc}.

\begin{Def}
Let $\alpha\in\widehat{K}$. We denote by $W_{\alpha}$ the ring of rational functions over $K$ which are integer-valued over $\alpha$, namely:
$$W_{\alpha}=\{\phi\in K(X) \mid v(\phi(\alpha))\geq0\}$$
\end{Def}
It is straightforward to verify that $W_{\alpha}$ is a valuation domain of $K(X)$ which lies over $V$ (see also \cite[Proposition 2.2]{PerTransc}).

\begin{Rem}\label{remark:BrE=0}
In case $E$ is a pseudo-convergent sequence with zero breadth ideal, and $\alpha\in \widehat K$ is the (unique) limit of $E$, since rational functions are continuous in the topology induced by $v$, we have $V_E=W_{\alpha}$.
 Moreover, $\alpha$ is algebraic (transcendental, respectively) over $K$ if and only if $E$ is of algebraic (transcendental, respectively) type.  These kind of valuations domains have been characterized in \cite[Proposition 2.2]{PerTransc}. We will deal with the case of non-zero breadth ideal in Theorem \ref{teor:rank-VE}.
\end{Rem}

We conclude this section by describing the valuation $v_E$, its residue field and its value group when $E$ is a pseudo-convergent sequence of algebraic type.
\begin{Prop}\label{description v_E}
Let $E=\{s_n\}_{n\in\N}\subset K$ be a pseudo-convergent sequence and suppose that $\beta\in K$ is a pseudo-limit of $E$; let $\alpha\in K$. Then, the following hold.
\begin{enumerate}[(a)]
\item $v_{E}(X-\alpha)\leq v_{E}(X-\beta)$ and equality holds if and only if $\alpha\in\limiti_E$.
\end{enumerate}
Let $\vEXbeta_E=v_{E}(X-\beta)\in\Gamma_{v_E}$ (which by above does not depend on the  choice of the pseudo-limit $\beta$ of $E$).
\begin{enumerate}[(a),resume]
\item $\vEXbeta_{E}$ is not a torsion element in $\Gamma_{v_E}/\Gamma_v$ (i.e., if $k\in\N$ is such that $k\cdot \vEXbeta_{E}\in\Gamma_v$, then $k=0$). In particular, $v_E=v_{\beta,\Delta_{E}}$.
\item $\Gamma_{v_E}=\Z\vEXbeta_{E}\oplus\Gamma_v$  (as groups).
\item $V_E/M_E\cong V/M$.
\end{enumerate}
\end{Prop}
\begin{proof}
The condition $v_{E}(X-\alpha)\leq v_{E}(X-\beta)$ is equivalent to $\phi(X)=\frac{X-\beta}{X-\alpha}\in V_{E}\Leftrightarrow\phi(s_n)\in V$, for almost all $n\in\N$. For each $n\in\N$, we have:
$$v(\phi(s_n))=v(s_n-\beta)-v(s_n-\alpha)$$
Now, we write $v(s_n-\alpha)=v(s_n-\beta+\beta-\alpha)$. Note that $\beta-\alpha\in\Br(E)\Leftrightarrow \alpha\in\mathcal{L}_E$ (Lemma \ref{lemma:kaplansky}). If these conditions hold, then $v(\beta-\alpha)>v(s_{n+1}-s_n)=v(s_n-\beta)$ for each $n\in\N$ and therefore $v(s_n-\alpha)=v(s_n-\beta)$. Note that in this case $\phi\in V_{E}^*$ and so, in particular, $\vEXbeta_E=v_E(X-\beta)$ does not depend on the pseudo-limit $\beta$ of $E$ we have chosen (in $K$). If instead $\alpha\not\in\mathcal{L}$ then there exists $N\in\N$ such that $v(\beta-\alpha)<v(s_{n+1}-s_n)=v(\beta-s_n)$ for all $n\geq N$. Hence, $v(s_n-\alpha)=v(\beta-\alpha)<v(\beta-s_n)$ for all $n\geq N$, that is, $\phi\in M_{E}\subset V_{E}$. 

We prove now the other three claims. Suppose there exists $k\in\N$ such that $k\cdot \vEXbeta_E\in\Gamma_v$, that is, $k\cdot v_E(X-\beta)=v_E((X-\beta)^k)=v(a)$, for some $a\in K$. This implies that $\frac{(X-\beta)^k}{a}\in V_E^*$, which is a contradiction, since $k\cdot v(s_n-\beta)-v(a)$ is strictly increasing.

Since $\vEXbeta_E=v_E(X-\beta)\in\Gamma_{v_E}$ is not torsion over $\Gamma_v$, by \cite[Chapt. VI, \S 10, Proposition 1]{Bourb} (see also \cite[p. 289]{APZ3}) we have that for each $f\in K[X]$, $f(X)=a_0+a_1(X-\beta)+\ldots+a_n(X-\beta)^n$, 
$$v_E(f(X))=\inf\{v(a_i)+i\vEXbeta_{ E} \mid i=0,\ldots,n\}$$
(where the $\inf$ is in $\Gamma_{v_E}$). In fact, we have $v_E(a_i(X-\beta)^i)\not=v_E(a_j(X-\beta)^j)$, for all $i\not=j\in\{0,\ldots,n\}$, otherwise $(i-j)\vEXbeta_{ E}=v(a_j)-v(a_i)$ and $\vEXbeta_{E}$ would be torsion over $\Gamma_v$. This implies that $v_E=v_{\beta,\vEXbeta_{E}}$. Moreover, by the same reference, $\Gamma_{v_E}=\Z\vEXbeta_{E}\oplus\Gamma_v$ and the residue field of $V_E$ is isomorphic to the residue field of $V$.
\end{proof} 

In the general case, where $E$ is algebraic but has no pseudo-limits in $K$, we only need to pass to an extension of $V$.
\begin{Cor}
Let $E\subset K$ be a pseudo-convergent sequence of algebraic type, and let $u$ be an extension of $v$ to $\overline{K}$. Let $\beta\in\overline K$ be a pseudo-limit of $E$ with respect to $u$, and let $\vEXbeta=u_E(X-\beta)$. Then, $v_E$ is equal to the restriction to $K(X)$ of $u_E=u_{\beta,\vEXbeta}$.
\end{Cor}
\begin{proof}
Let $U_E$ be the valuation domain of $\overline{K}(X)$ associated to $E$ with respect to $u$. By Proposition \ref{description v_E}, $u_E=u_{\beta,\vEXbeta}$, where $\vEXbeta=u_E(X-\beta)$. Since $U_E\cap K(X)=V_E$, the claim follows immediately.
\end{proof}

\section{The rank of $V_E$ and the Ostrowski valuation $w_E$}\label{sect:WE}

We assume for the rest of the article that $V$ has rank one.

If $W$ is an extension of $V$ to $K(X)$, then the rank of $W$ is either $1$ or $2$ (\cite[Chapitre VI, \S 10, Corollaire 1, p. 162]{Bourb}). In this section, we want to determine when each of the two possibilities occurs for $W=V_E$, where $E$ is a pseudo-convergent sequence. To this end, we need to introduce another kind of valuation on $K(X)$ which lies over $V$; also this valuation arise from pseudo-convergent sequences and have been first introduced and studied by Ostrowski in \cite[65. p. 374]{Ostr}. 

\begin{Def}
Let $E=\{s_n\}_{n\inN}\subset K$ be a pseudo-convergent sequence. We define $w_E$ as the map
\begin{equation*}
\begin{aligned}
w_E\colon K(X) & \longrightarrow \insR\cup\{\pm\infty\}\\
\phi & \longmapsto \lim_{n\to\infty}v(\phi(s_n)).
\end{aligned}
\end{equation*}
If $w_E$ happens to define a valuation on $K(X)$, we denote by $W_E$ the associated valuation domain, namely, $W_E=\{\phi\in K(X)\mid w_E(\phi)\geq0\}$.
\end{Def}
Note that, for large $n$, $s_n$ is neither a zero nor a pole of $\phi$, so $v(\phi(s_n))$ is defined for all large $n$. We are going to show under which cases $w_E$ is a valuation on $K(X)$.

One of the main accomplishments of Ostrowski (and also the motivation for the introduction of the notion of pseudo-convergent sequence) in his work \cite{Ostr} is the \emph{Fundamentalsatz}, which we now recall.
\begin{teor}
\cite[66. IX, p. 378]{Ostr} Let $K$ be an algebraically closed field and let $v$ be a rank one valuation on $K$. If $w$ is a rank one valuation of $K(X)$ extending $v$, then there is a pseudo-convergent sequence $E=\{s_n\}_{n\inN}\subset\overline{K}$ such that $w=w_E$.
\end{teor}
When $K$ is not algebraically closed, this means that the rank one valuations of $K(X)$ extending $v$ can be realized as the contraction to $K(X)$ of the valuations $w_E$ on $\overline K(X)$ for some pseudo-convergent sequence $E\subset\overline{K}$ and some extension of $v$ to $\overline K$.

For the sake of completeness, in the next two propositions we prove the basic properties of the function $w_E$.
\begin{prop}\label{prop:limsup}
Let $E=\{s_n\}_{n\inN}$ be a pseudo-convergent sequence  that is either of transcendental type or of algebraic type and non-zero breadth ideal. Then the map $w_E:K(X)\to \R\cup\{\infty\}$  extends $v$ and is a valuation of rank one on $K(X)$. Furthermore, the valuation ring $W_E$ relative to $w_E$ contains $V_E$.
\end{prop}
\begin{proof}
Suppose first that $E$ is of transcendental type. Then for each $\phi\in K(X)$, $v(\phi(s_n))$ is eventually constant, and furthermore $w_E(\phi)=\infty$ if and only if $\phi=0$. 

Suppose now that $E$ is of algebraic type and the breadth ideal is non-zero. Then also in this case $w_E$ is well-defined, since by Proposition \ref{prop:qsn} for every $\phi\in K(X)$ there is a $k\inZ$ and $\gamma\in\Gamma_v$ such that $v(\phi(s_n))=k\delta_n+\gamma$, and $\delta_n\to\delta$ as $n\to\infty$. Moreover, since $\delta<\infty$, and since $\phi$ has only finitely many zeros and points where it is not defined, we have $w_E(\phi)=\infty$ if and only if $\phi=0$. 

In either case, if $\phi=a\in K$ is a constant, then $w_E(\phi)=v(a)$; thus, $w_E$ extends $v$.

If now $\phi_1,\phi_2\in K(X)$ then
\begin{equation*}
v((\phi_1+\phi_2)(s_n))=v(\phi_1(s_n)+\phi_2(s_n))\geq\min\{v(\phi_1(s_n)),v(\phi_2(s_n))\};
\end{equation*}
hence, $w_E(\phi_1+\phi_2)\geq\min\{w_E(\phi_1),w_E(\phi_2)\}$. In the same way, $w_E(\phi_1\phi_2)=w_E(\phi_1)+w_E(\phi_2)$. Hence, $w_E$ is a valuation.

If now $\phi\in V_E$, then $\phi(s_n)\in V$ for large $n$, or equivalently $v(\phi(s_n))\geq 0$ for large $n$. In particular, $\lim v(\phi(s_n))\geq 0$, i.e., $w_E(\phi)\geq 0$. Therefore, $\phi\in W_E$.
\end{proof}

If $E$ is of algebraic type and its breadth ideal is zero, on the other hand, $w_E$ is not a valuation: this is due to the fact that $w_E(\phi)$ tends to $\infty$ when the pseudo-limit of $E$ (in $\overline{K}$) is a zero of $\phi$. It is, however, very close to a valuation: recall that a \emph{pseudo-valuation} of a field $K$ is a map $v$ from $K$ to $\Gamma_v\cup\{\infty\}$, where $\Gamma_v$ is a totally ordered abelian group, which satisfies the same axioms of a valuation except that we are not assuming that $v(x)=\infty\Rightarrow x=0$. The set $\{x\in K\mid v(x)=\infty\}$ is a prime ideal of the valuation domain $V$ of $v$, called the \emph{socle} of $v$. The following proposition is straightforward.
\begin{prop}\label{prop:pseudoval}
Let $E=\{s_n\}_{n\in\N}\subset K$ be a pseudo-convergent sequence of algebraic type and zero breadth ideal. If $q\in K[X]$ is the minimal polynomial of the limit of $E$ in $\widehat{K}$, then the map $w_E:K[X]_{(q)}\to \R\cup\{\infty\}$  extends $v$ and is a pseudo-valuation with socle $q(X)K[X]_{(q)}$. Moreover, the valuation ring of $w_E$, that is, $\{\phi\in K[X]_{(q)}\mid w_E(\phi)\geq0\}$, is equal to $V_E$.
\end{prop}

\begin{Def}
Let $E\subset K$ be a pseudo-convergent sequence which is either of transcendental type or of algebraic type and non-zero breadth ideal. We call the associated rank one valuation $w_E:K(X)\to\R\cup\{\infty\}$ the \emph{Ostrowski valuation} associated to $E$, and the corresponding valuation domain $W_E$ the \emph{Ostrowski valuation domain}  associated to $E$.
\end{Def}


The following corollary is an easy consequence of Proposition \ref{prop:qsn}. It also follows from Lemma \ref{VII} below.
\begin{cor}\label{cor:wEphi}
Let $E=\{s_n\}_{n\inN}\subset K$ be a pseudo-convergent sequence of algebraic type with breadth $\delta$, and let $\phi\in K(X)$.  If $\lambda=\degdom_E(\phi)$ and $\gamma\in\Gamma_v$ is as in Proposition \ref{prop:qsn}, then we have
\begin{equation}\label{wEphi}
w_E(\phi)=\lambda\delta+\gamma
\end{equation}
In particular, $\Gamma_{w_E}=\Z\delta+\Gamma_v$.
\end{cor}

\begin{Rem}\label{Rem:osservazioni su wE}
\begin{enumerate}[(a)]
\item\label{Rem:osservazioni su wE:X} Let $u$ be an extension of $v$ to $\overline 
K$. It follows at once from Corollary \ref{cor:wEphi} that if $E=\{s_n\}_{n\inN}\subset K$ is a pseudo-convergent sequence of algebraic type with breadth $\delta$ and $\beta\in\mathcal{L}_E^u$, then, for each $s\in K$ we have:
\begin{equation}\label{wE(X-beta)}
w_E(X-s)=\lim_{n\to\infty}v(s_n-s)=
\left\{
\begin{array}{cc}
\delta,&\; \textnormal{ if }s\in\mathcal{L}_E\\
u(s-\beta)<\delta,&\textnormal{ if }s\notin\mathcal{L}_E
\end{array}
\right.
\end{equation}
Note that, in case $s\notin\mathcal{L}_E$ and $\beta'\in\mathcal{L}_E^u$, we have $w_E(X-s)=u(s-\beta)=u(s-\beta')$, thus this value is independent of the chosen pseudo-limit of $E$. Similarly, if $E$ is of transcendental type, then $w_E(X-s)<\delta$ for any $s\in K$.
\item Under the assumption of Corollary \ref{cor:wEphi}, let $u$ be a fixed extension of $v$ to $\overline{K}$ and let $w_E$ be extended to $\overline{K}(X)$ along $u$ (i.e., $w_E(\psi)=\lim u(\psi(s_n))$, for any $\psi\in\overline{K}(X)$). Let $S$ be the multiset of critical points of $\phi$ which are pseudo-limits of $E$. Then by (\ref{wEphi}) and (\ref{wE(X-beta)}) we have
$$w_E(\phi)=w_E(\phi_S)+w_E\left(\frac{\phi}{\phi_S}\right)=\lambda\delta+\gamma$$
where $\lambda\delta=w_E(\phi_S)$ and $w_E\left(\frac{\phi}{\phi_S}\right)=u\left(\frac{\phi}{\phi_S}(\beta)\right)=\gamma$, where $\beta\in\mathcal{L}_E^u$, since $w_E(X-\alpha)=v(\beta-\alpha)$ for every $\alpha\notin S$ by the previous remark. We stress the strong analogy between this expression of $w_E$ and the valuation associated to a valuation domain of the form $W_{\alpha}$, for $\alpha\in\widehat{K}$ which is algebraic over $K$. See \cite[p. 126]{CaCh} and \cite[Remark 2.3]{PerTransc}.
\end{enumerate}
\end{Rem}

The next lemma is taken from \cite{Ostr} and gives an important connection between monomial valuations and Ostrowski valuations; we repeat it here for the convenience of the reader.
\begin{lemma}\cite[VII, p. 377]{Ostr}\label{VII}
Let $E=\{s_n\}_{n\in\N}$ be  a pseudo-convergent sequence of algebraic type and let $\alpha\in\mathcal{L}_E^u$, for some extension $u$ of $v$ to $\overline{K}$. Then $w_E=(u_{\alpha,\delta_E})_{|K(X)}=v_{\alpha,\delta_E}$.
\end{lemma}
\begin{proof}
We can reduce to proving the statement when $K$ is algebraically closed, so in particular $u=v$. We have to show that $w_E=v_{\alpha,\delta}$, where $\delta=\delta_E$ and $\alpha\in K$ is a pseudo-limit of $E$. Let $\beta\in K$. To this end, by \cite[IV, p. 366]{Ostr}, it is sufficient to show that $w_E(X-\alpha+\beta)=\min\{w_E(X-\alpha),v(\beta)\}=\{\delta,v(\beta)\}$. If $\delta\not=v(\beta)$ then this is clear, so suppose that $\delta=v(\beta)$. We have:
$$w_E(X-\alpha+\beta)=\lim_{n\to\infty}v(s_n-\alpha+\beta)=\lim_{n\to\infty}v(s_n-\alpha)=\delta$$
so that also in this case we have the claimed equality.
\end{proof}

We now show under which cases $V_E$ has rank $1$ or $2$.
\begin{teor}\label{teor:rank-VE}
Let $E\subset K$ be a pseudo-convergent sequence.
\begin{enumerate}[(a)]
\item\label{teor:rank-VE:trasc} If $E$ is of transcendental type, then $V_E=W_E$ has rank 1.
\item\label{teor:rank-VE:alg-inf} If $E$ is of algebraic type and its breadth is infinite, then $V_E$ has rank 2; furthermore, if $q$ is the minimal polynomial of the pseudo-limit of $E$, then the one-dimensional overring of $V_E$ is $K[X]_{(q)}$.
\item\label{teor:rank-VE:alg-fin} Suppose that $E$ is of algebraic type with breadth $\delta\inR$. The following conditions are equivalent:
\begin{enumerate}[(i)]
\item\label{teor:rank-VE:alg-fin:tors} $\delta$ is not torsion over $\Gamma_v$;
\item\label{teor:rank-VE:alg-fin:rtrasc} $W_E$ is not residually transcendental over $V$;
\item\label{teor:rank-VE:alg-fin:rank} $V_E$ has rank one;
\item\label{teor:rank-VE:alg-fin:ug} $V_E=W_E$;
\item\label{teor:rank-VE:alg-fin:trace} $V_E\cap K[X]=W_E\cap K[X]$.
\end{enumerate}
\end{enumerate}
\end{teor}
\begin{proof}
\ref{teor:rank-VE:trasc} follows directly from \cite[Theorem 2]{Kap} and the proof of Theorem \ref{Thm:VE valuation domain}, while \ref{teor:rank-VE:alg-inf} is a direct consequence of Remark \ref{remark:BrE=0} (and Proposition \ref{prop:pseudoval}).

\ref{teor:rank-VE:alg-fin} Let $E=\{s_n\}_{n\inN}$ be of algebraic type with finite breadth, and let $\{\delta_n\}_{n\inN}$ be the gauge of $E$. Since $V_E\subseteq W_E$ and $W_E$ has rank 1, conditions \ref{teor:rank-VE:alg-fin:rank} and \ref{teor:rank-VE:alg-fin:ug} are clearly equivalent. Since by Lemma \ref{VII} $w_E=v_{\alpha,\delta}$, by \cite[Lemma 3.5]{PerPrufer} we have that \ref{teor:rank-VE:alg-fin:tors} is equivalent to \ref{teor:rank-VE:alg-fin:rtrasc}. Clearly, \ref{teor:rank-VE:alg-fin:ug} implies \ref{teor:rank-VE:alg-fin:trace}.

\ref{teor:rank-VE:alg-fin:trace} $\Longrightarrow$ \ref{teor:rank-VE:alg-fin:ug} Let $\phi\in W_E$ , i.e., $w_E(\phi)\geq 0$. Clearly, if $w_E(\phi)>0$ then $\phi\in V_E$, so suppose $w_E(\phi)=0$. By Proposition \ref{prop:qsn} and Corollary \ref{cor:wEphi}, there exist $\lambda\in\Z$ and $\gamma\in\Gamma_v$ such that $v(\phi(s_n))=\lambda\delta_n+\gamma$ for all large $n$; its limit $\lambda\delta+\gamma$ is equal to $w_E(\phi)$, and thus it is $0$. If $\lambda\leq 0$, then $v(\phi(s_n))$ is eventually positive, and $\phi\in V_E$.

Suppose $\lambda>0$ and let $p\in K[X]$ be the minimal polynomial of some pseudo-limit $\beta$ of $E$ (with respect to some extension $u$ of $v$ in $\overline{K}$). By Corollary \ref{cor:wEphi}, there are $\lambda_p\inZ$, $\gamma_p\in\Gamma_v$ such that $v(p(s_n))=\lambda_p\delta_n+\gamma_p$ for all large $n$; furthermore, $\lambda_p>0$ since $p$ is a polynomial and one of the roots of $p(X)$ is a pseudo-limit of $E$. Let $c\in K$ be an element of value $\lambda_p\gamma-\lambda\gamma_p$ (which exists since $\lambda,\lambda_p\in\Z$ and $\gamma,\gamma_p\in\Gamma_v$) and consider $\psi(X)=c p(X)^{\lambda}\in K[X]$. Then, for all $n\in\N$ sufficiently large we have
\begin{equation*}
v(\psi(s_n))=\lambda_p\gamma-\lambda\gamma_p+\lambda(\lambda_p\delta_n+\gamma_p)=\lambda_p(\gamma+\lambda\delta_n)=\lambda_pv(\phi(s_n)).
\end{equation*}
This quantity has limit $0$ as $n\to\infty$ and is strictly increasing, because $\lambda\lambda_p>0$; hence $v(\psi(s_n))<0$ for all $n\in\N$ sufficiently large. Therefore, $\psi\in W_E\setminus V_E$. However, this contradicts the hypothesis because $\psi(X)$ is a polynomial; hence, the claim is proved.

We show now that \ref{teor:rank-VE:alg-fin:tors}  $\iff$ \ref{teor:rank-VE:alg-fin:ug}, and we claim that it is sufficient to prove  the equivalence under the further assumption that $E$ has a pseudo-limit $\beta$ in $K$. Suppose that \ref{teor:rank-VE:alg-fin:tors} is equivalent to \ref{teor:rank-VE:alg-fin:ug} under this assumption and let $\beta\in\mathcal{L}_{E}^{v'}$ where $v'$ is an extension of $v$ to $K(\beta)$. Let $V_E'\subseteq W_E'$ be the valuation domains of $K(\beta)(X)$ associated to $E$ with respect to the valuation $v'$. If $\delta$ is not torsion over $\Gamma_v$ then $V_E'=W_E'$ and contracting down to $K(X)$ we get $V_E=W_E$. Conversely, if the rank of $V_E$ is one (thus, $V_E=W_E$) then also the rank of $V_E'$ is one (because $K(X)\subseteq K(\beta)(X)$ is an algebraic extension) and so $\delta$ is not torsion over $\Gamma_v$.

Suppose thus that $\beta\in K$ is a pseudo-limit of $E$. By \eqref{wE(X-beta)} we have $w_E(X-\beta)=\delta$. If $\delta$ is torsion over $\Gamma_v$, then $k\delta\in\Gamma_v$ for some $k\in\N$, i.e., there is $c\in K$ such that $w_E((X-\beta)^k)=v(c)$; let $\phi(X)=\frac{(X-\beta)^k}{c}$. Then, $w_E(\phi)=0$ and thus $\phi\in W_E$, while
\begin{equation}\label{vE(X-beta)^k/c}
v\left(\frac{(s_n-\beta)^k}{c}\right)=k\delta_n-v(c)<0,
\end{equation}
and thus $\phi(s_n)\notin V$ for every $n$, which implies $\phi\notin V_E$. Hence, $V_E\neq W_E$.

Conversely, suppose that $\delta$ is not torsion over $\Gamma_v$, and let $\phi\in W_E$. If $w_E(\phi)>0$ then $\phi$ belongs to the maximal ideal of $W_E$, which is contained in $V_E$. Suppose $w_E(\phi)=0$, and let $k$ be the dominating degree of $E$. By definition we have $w_E(\phi)=k\delta+\gamma$ for some $\gamma\in\Gamma_v$ (see also Corollary \ref{cor:wEphi} and (\ref{wEphi})); since this quantity is 0 and $\delta$ is not torsion, we must have $k=0$, and so also $\gamma=0$. But this means that $v(\phi(s_n))=0$ for large $n$; in particular, $\phi(s_n)\in V$ for large $n$. Thus $\phi\in V_E$ and $V_E=W_E$.
\end{proof}


When the rank of $V_E$ is 1, then its valuation $v_E$ is exactly the Ostrowski valuation $w_E$; on the other hand, if $E$ is algebraic with infinite breadth, then $V_E$ has been characterized in Remark \ref{remark:BrE=0} and its valuation is described in \cite[Remark 2.3]{PerTransc}. When $V_E$ has rank 2 and $E$ has finite breadth, a description of $v_E$ has been obtained in Proposition \ref{description v_E}; we now want to embed $\Gamma_{v_E}$ as a totally ordered subgroup in $\insR^2$, endowed with the lexicographic order (this can be done by Hahn's theorem \cite[Th\'eor\`eme 2, p. 22]{Ribenboim}).

\begin{teor}\label{teor:valut-dim2}
Let $E=\{s_n\}_{n\inN}$ be a pseudo-convergent sequence with non-zero breadth  ideal such that $V_E$ has rank 2. Then the map
\begin{equation*}
\begin{aligned}
\nu\colon K(X)\setminus\{0\}& \longrightarrow  \R^2\\
\phi & \longmapsto (w_E(\phi),-\degdom_E(\phi))
\end{aligned}
\end{equation*}
is a valuation on $K(X)$ whose associated valuation ring is $V_E$.
\end{teor}
\begin{proof}
By Theorem \ref{teor:rank-VE}, $E$ is of algebraic type and its breadth $\delta$ is torsion over $\Gamma_v$. Let $\{\delta_n\}_{n\in\N}$ be the gauge of $E$. Since $w_E$ is a valuation, we have $w_E(\phi_1\phi_2)=w_E(\phi_1)+w_E(\phi_2)$ for every $\phi_1,\phi_2\in K(X)$; the same formula holds for the dominating degree, since the multiset of zeros of $\phi_1\phi_2$ is exactly the union of the multisets of zeros of $\phi_1$ and $\phi_2$. Hence, $\nu(\phi_1\phi_2)=\nu(\phi_1)+\nu(\phi_2)$.

We now want to show that $\nu(\phi_1+\phi_2)\geq\min\{\nu(\phi_1),\nu(\phi_2)\}$. Let $\lambda_1=\degdom_E(\phi_1)$, $\lambda_2=\degdom_E(\phi_2)$, $\lambda=\degdom_E(\phi_1+\phi_2)$. By Proposition \ref{prop:qsn}, there are $\gamma_1,\gamma_2,\gamma\in\Gamma_v$ such that $v(\phi_i(s_n))=\lambda_i\delta_n+\gamma_i$, $i=1,2$ and $v((\phi_1+\phi_2)(s_n))=\lambda\delta_n+\gamma$ for all large $n$. Furthermore, by Corollary \ref{cor:wEphi}, $w_E(\phi_i)=\lambda_i\delta+\gamma_i$, $i=1,2$ and $w_E(\phi_1+\phi_2)=\lambda\delta+\gamma$.

We distinguish four cases.

If $w_E(\phi_1)\neq w_E(\phi_2)$, then without loss of generality $w_E(\phi_1)<w_E(\phi_2)$.  Hence, we have $v(\phi_1(s_n))<v(\phi_2(s_n))$ for all large $n$. Thus,
\begin{equation*}
v((\phi_1+\phi_2)(s_n))=v(\phi_1(s_n))=\lambda_1\delta_n+\gamma_1.
\end{equation*}
Hence, $\lambda_1\delta_n+\gamma_1=\lambda\delta_n+\gamma$ infinitely many times. Thus, it must be $\lambda_1=\lambda$, and so $\nu(\phi_1+\phi_2)=\nu(\phi_1)=\min\{\nu(\phi_1),\nu(\phi_2)\}$.

If $w_E(\phi_1)=w_E(\phi_2)<w_E(\phi_1+\phi_2)$, then $\nu(\phi_1+\phi_2)$ is bigger than both $\nu(\phi_1)$ and $\nu(\phi_2)$, and we are done.

Suppose that $w_E(\phi_1)=w_E(\phi_2)=w_E(\phi_1+\phi_2)$ and that $\lambda_1\neq\lambda_2$; without loss of generality, $\lambda_1>\lambda_2$ (i.e., $\nu(\phi_1)<\nu(\phi_2)$). Then,  $\lambda_1\delta_n+\gamma_1<\lambda_2\delta_n+\gamma_2$ for all large $n$. Therefore, as in the first case, $\lambda_1\delta_n+\gamma_1=\lambda\delta_n+\gamma$ for all large $n$, and so $\nu(\phi_1+\phi_2)=\nu(\phi_1)$.

Suppose now that $w_E(\phi_1)=w_E(\phi_2)=w_E(\phi_1+\phi_2)$ and that $\lambda_1=\lambda_2=:\lambda'$. Since the sequences $v(\phi_1(s_n))=\lambda'\delta_n+\gamma_1$ and $v(\phi_2(s_n))=\lambda'\delta_n+\gamma_2$ have the same limit, they must be eventually equal, and so $\gamma_1=\gamma_2=:\gamma'$. Since $v$ is a valuation, $v((\phi_1+\phi_2)(s_n))=\lambda\delta_n+\gamma\geq\lambda'\delta_n+\gamma'$. Since $w_E(\phi_1+\phi_2)=w_E(\phi_1)$, furthermore, the limits $\lambda\delta+\gamma$ and $\lambda'\delta+\gamma'$ are equal; it follows that $\lambda\leq\lambda'$. Hence,
\begin{equation*}
\nu(\phi_1+\phi_2)=(w_E(\phi_1),-\lambda)\geq(w_E(\phi_1),-\lambda')=\nu(\phi_1)=\nu(\phi_2).
\end{equation*}
Therefore, $\nu$ is a valuation.

The fact that $\nu$ extends $v$ follows from the fact that $w_E$ extends $v$.

Let $V'$ be the valuation ring associated to $\nu$. Suppose $\phi\in V_E$. If $w_E(\phi)>0$ then $\nu(\phi)>0$. If $w_E(\phi)=0$ then $v(\phi(s_n))=\lambda\delta_n+\gamma$ must tend to 0 from above, and thus $\lambda\leq 0$, i.e., $\nu(\phi)\geq 0$. Thus, $V_E\subseteq V'$. Conversely, if $\nu(\phi)\geq 0$ then either $w_E(\phi)>0$ (and so $\phi\in M_{W_E}\subset V_E$) or $w_E(\phi)=0$ and $\lambda\leq 0$; in the latter case, $\lambda\delta_n+\gamma\geq 0$, and so $v(\phi(s_n))\geq 0$. Therefore, $V'\subseteq V_E$, and so $V'=V_E$, as claimed.
\end{proof}

\begin{Rem}
Let $E$ be a pseudo-convergent sequence of algebraic type with breadth $\delta$ which is torsion over $\Gamma_v$. Since $\nu$ is a valuation relative to $V_E$, without loss of generality we can set $v_E=\nu$.

Let $\vEXbeta'=(\delta,-1)$. Take $\phi\in K(X)$ and let $\lambda=\degdom_E(\phi)$. By Theorem \ref{teor:valut-dim2}, we have
\begin{equation*}
v_E(\phi)=(w_E(\phi),-\lambda);
\end{equation*}
by Corollary \ref{cor:wEphi}, moreover, $w_E(\phi)=\lambda\delta+\gamma_x$ for some $\gamma_x\in\Gamma_v$. It follows that
\begin{equation*}
v_E(\phi)=\lambda\vEXbeta'+\gamma,
\end{equation*}
where $\gamma=(\gamma_x,0)\in\Gamma_{v_E}$. If, furthermore, $E$ has a pseudo-limit $\beta\in K$, then $\vEXbeta'=\vEXbeta=v_E(X-\beta)$.
\end{Rem}

\section{Equivalence of pseudo-convergent sequences}\label{sect:equiv}

We recall that $V$ is a rank one valuation domain.

Classically, two Cauchy sequences $E,F\subset K$ are equivalent if the distance induced by the valuation $v$ between their corresponding terms goes to zero. If $\alpha$ and $\beta$ are the limits in $\widehat K$ of $E$ and $F$, respectively, it is known that $E$ and $F$ are equivalent if and only if the valuation domains $V_E=W_{\alpha}, V_F=W_{\beta}$ (see Remark \ref{remark:BrE=0}) are the same; in particular, $E$ and $F$ determine the same extension of the valuation $v$ to $K(X)$.  Ostrowski investigated in  \cite[p. 387]{Ostr} the similar problem for the valuation domains of the form $W_E$, for $E$ a pseudo-convergent sequence in $K$, which led him to give the notion of equivalent pseudo-convergent sequences. In this section, we consider a definition of equivalence for pseudo-convergent sequence as it appears in \cite[Section 3.2]{KarpWah}, even though we correct a mistake there. 


\begin{Def}\label{defin:equiv}
Let $E=\{s_n\}_{n\inN}$ and $F=\{t_n\}_{n\inN}$ be two pseudo-convergent sequences in $K$. We say that $E$ and $F$ are \emph{equivalent} if the breadths $\delta_E$ and $\delta_F$ are equal and, for every $k\inN$, there are $i_0,j_0\inN$ such that, whenever $i\geq i_0$, $j\geq j_0$, we have
\begin{equation*}
v(s_i-t_j)>v(t_{k+1}-t_k).
\end{equation*}
\end{Def} 
Note that the previous definition boils down to the classical notion of equivalence if $E$ and $F$ are Cauchy sequences.
\begin{oss}\label{mistake definition equivalence}
The previous definition was also considered in \cite[Section 3.2]{KarpWah} without the hypothesis $\delta_E=\delta_F$. However, without this condition the definition is not symmetric: for example, let $F=\{t_n\}_{n\inN}$ be a sequence in $V$ with $v(t_n)=\delta_n$, where $\{\delta_n\}_{n\inN}$ is a  positive increasing sequence, and let $E=\{s_n=t_n^2\}_{n\inN}$. Then, for every $k$ and every $i,j\geq k+1$ we have
\begin{equation*}
v(s_i-t_j)=\delta_j>\delta_k=v(t_{k+1}-t_k);
\end{equation*}
on the other hand, if $\delta_k>\inv{2}\delta$, then there are no $i,j$ such that
\begin{equation*}
v(s_i-t_j)>2\delta_k=v(s_{k+1}-s_k);
\end{equation*}
hence, $E$ and $F$ are equivalent according to \cite{KarpWah}, but $F$ and $E$ are not.

On the other hand, suppose that $E$ and $F$ are two pseudo-convergent sequence of $K$ which are equivalent according to Definition \ref{defin:equiv}. Then, for every $k$ there is a $k'$ such that $v(s_{k+1}-s_k)<v(t_{k'+1}-t_{k'})$. If now $i_0$ and $j_0$ are such that $v(s_i-t_j)>v(t_{k'+1}-t_{k'})$ for all $i\geq i_0$, $j\geq j_0$, then clearly $v(s_i-t_j)>v(s_{k+1}-s_k)$, so $F$ and $E$ are equivalent.
\end{oss}

We need first the following preliminary lemma.
\begin{lemma}\label{bothtrans bothalgebr}
Let $E,F\subset K$ two equivalent pseudo-convergent sequences. Then either $E$ and $F$ are both of transcendental type, or $E$ and $F$ are both of algebraic type. In the latter case, {$\limiti_E^u=\limiti_F^u$ for every} extension $u$ of $v$ to $\overline{K}$.
\end{lemma}
\begin{proof}
Let $E=\{s_n\}_{n\inN}$  and $F=\{t_n\}_{n\inN}$; let $\{\delta_n\}_{n\in\N},\{\delta'_n\}_{n\in\N}$  be the gauges of $E$ and $F$, respectively, and $\delta$ the breadth of $E$ and $F$. It is sufficient to prove that if either one of the two pseudo-convergent sequences, say $E$, is of algebraic type, then also the other is of algebraic type.

Suppose first that $K$ is algebraically closed and let $\beta$ be a pseudo-limit of $E$. Fix $k\in\N$. Then there exist $i_0,j_0\in\N$ such that, for all $m\geq i_0, n\geq j_0$, $v(s_m-t_n)>\delta_k$. For such $n$ and $m$, suppose also that $m\geq k$. Then $v(t_n-\beta)=v(t_n-s_m+s_m-\beta)\geq \delta_k$. Therefore, $w_F(X-\beta)=\lim_{n\to\infty} v(t_n-\beta)\geq \delta$.
If $w_F(X-\beta)>\delta$, then there is a $n_0$ such that, for all $n\geq n_0$, $v(t_n-\beta)>\delta$;  since $v(s_m-\beta)=\delta_m<\delta$, this means that, for every $m$ sufficiently large, $v(t_n-s_m)=v(s_m-\beta)$. This would imply that $t_n$ is a pseudo-limit of $E$ for all $n\geq n_0$, and thus that, in particular, $v(t_{n+1}-t_n)\geq\delta$, which is a contradiction since $v(t_{n+1}-t_n)=\delta'_n\nearrow\delta$.  Hence, $\delta_n\leq v(t_n-\beta)<\delta$ for all large $n$ and $w_F(X-\beta)=\delta$; this shows that $v(t_n-\beta)$ is eventually strictly increasing, that is, $\beta$ is a pseudo-limit of $F$ and thus also $F$ is of algebraic type. Moreover, since $\beta\in\mathcal{L}_E$ was arbitrarily chosen, we also have $\mathcal{L}_E\subseteq\mathcal{L}_F$, which shows that these sets are equal since they are closed balls of the same radius (Lemma \ref{lemma:kaplansky}). 

If now $K$ is not algebraically closed, let $u$ be any extension of $v$ to $\overline{K}$. Then, $E$ and $F$ are equivalent with respect to $u$; applying the previous part of the proof, we have that $F$ is of algebraic type and $\limiti_E^u=\limiti_F^u$, as claimed.
\end{proof}

\begin{teor}\label{teor:equiv-VE}
Let $E,F\subset K$ be two pseudo-convergent sequences  that are of transcendental type or of algebraic type with nonzero breadth ideal. Then, the following are equivalent:
\begin{enumerate}[(i),noitemsep]
\item\label{prop:equiv-VE:equiv} $E$ and $F$ are equivalent;
\item\label{prop:equiv-VE:V} $V_E=V_F$;
\item\label{prop:equiv-VE:W} $W_E=W_F$;
\item\label{prop:equiv-VE:w} $w_E=w_F$.
\end{enumerate}
Furthermore, if $E$ and $F$ are of algebraic type, the previous conditions are equivalent to the following:
\begin{enumerate}[(i),resume]
\item\label{prop:equiv-VE:limall} $\limiti_E^u=\limiti_F^u$ for all extensions $u$ of $v$ to $\overline{K}$;
\item\label{prop:equiv-VE:limun} $\limiti_E^u=\limiti_F^u$ for an extension $u$ of $v$ to $\overline{K}$.
\end{enumerate}
\end{teor}
\begin{proof}
As usual, we set $E=\{s_n\}_{n\inN}$  and $F=\{t_n\}_{n\inN}$; let $\{\delta_n\}_{n\in\N},\{\delta'_n\}_{n\in\N}$  be the gauges of $E$ and $F$, respectively, and $\delta, \delta'$ the breadths of $E$, $F$, respectively. Recall that, by Proposition \ref{prop:limsup}, if $E$ and $F$ are of transcendental type, then $V_E=W_E$ and $V_F=W_F$.

The structure of the proof is as follows:
\begin{itemize}
\item[-] we first prove \ref{prop:equiv-VE:V} $\Longrightarrow$ \ref{prop:equiv-VE:W} $\iff$ \ref{prop:equiv-VE:w} $\Longrightarrow$ \ref{prop:equiv-VE:equiv} in both the algebraic and the transcendental case;
\item[-] then we prove \ref{prop:equiv-VE:equiv} $\Longrightarrow$ \ref{prop:equiv-VE:w} and \ref{prop:equiv-VE:W} $\Longrightarrow$ \ref{prop:equiv-VE:V} in the transcendental case;
\item[-] finally, we prove \ref{prop:equiv-VE:equiv} $\Longrightarrow$ \ref{prop:equiv-VE:limall} $\Longrightarrow$ \ref{prop:equiv-VE:limun} $\Longrightarrow$ \ref{prop:equiv-VE:V} in the algebraic case.
\end{itemize}

\ref{prop:equiv-VE:V} $\Longrightarrow$ \ref{prop:equiv-VE:W} and \ref{prop:equiv-VE:w} $\Longrightarrow$ \ref{prop:equiv-VE:W} are obvious.

\ref{prop:equiv-VE:W} $\Longrightarrow$ \ref{prop:equiv-VE:w} Suppose there is a $\phi\in K(X)$ such that $w_E(\phi)\neq w_F(\phi)$; without loss of generality, $w_E(\phi)>w_F(\phi)$. We claim that there is a $c\in K$ such that $w_E(\phi)\geq v(c)>w_F(\phi)$. This is obvious if $\Gamma$ is dense in $\insR$; otherwise, $\Gamma$ must be isomorphic to $\insZ$, and $V$ is a discrete valuation ring. In this case, the breadth of $E$ and $F$ must be infinite, and thus (by hypothesis) $E$ and $F$ must be transcendental. However, by \cite[Theorem 2]{Kap}, it follows that $V_E=W_E$ is an immediate extension of $V$; in particular, the value group of $W_E$ coincide with $\Gamma$, and thus we can take a $c\in K$ such that $v(c)=w_E(\phi)$. The existence of $c$ implies that $\frac{\phi}{c}\in W_E$ while $\frac{\phi}{c}\notin W_F$, contradicting $W_E=W_F$. Hence, \ref{prop:equiv-VE:w} holds.

\ref{prop:equiv-VE:w} $\Longrightarrow$ \ref{prop:equiv-VE:equiv} By definition, for every $k$ and every $l> 0$,
\begin{equation}\label{deltak'=wEXtk}
\delta'_k=v(t_{k+l}-t_k)=\lim_{n\to\infty}v(t_n-t_k)=w_F(X-t_k)=w_E(X-t_k)=\lim_{n\to\infty}v(s_n-t_k).
\end{equation}
If $\delta_E<\delta_F$, then $\delta'_k>\delta_E$ for large $k$; thus,
\begin{equation*}
v(s_n-t_{k+1})=v(s_n-s_{n+1}+s_{n+1}-t_{k+1})=\delta_n
\end{equation*}
and thus $\delta_n=\delta'_{k+1}$, a contradiction; hence $\delta_E\geq\delta_F$. By symmetry, we have also $\delta_F\geq\delta_E$, and thus $\delta_E=\delta_F=\delta$.

Fix now $k$, take $n'_0$ such that $\delta_n>\delta'_k$ for every $n\geq n'_0$; since $\delta'_m>\delta'_k$ if $m>k$, there are $n_0>n'_0$ and $m_0>k$ such that $v(s_{n_0}-t_{m_0})>\delta'_k$. For all $n\geq n_0$, $m\geq m_0$, we have
\begin{equation*}
v(s_n-t_m)=v(s_n-s_{n_0}+s_{n_0}-t_{m_0}+t_{m_0}-t_m).
\end{equation*}
The three quantities $v(s_n-s_{n_0})$, $v(s_{n_0}-t_{m_0})$ and $v(t_{m_0}-t_m)$ are all bigger than $\delta'_k$; hence, so is $v(s_n-t_m)$. Since $k$ was arbitrary, $E$ and $F$ are equivalent.

Suppose now that $E$ is of transcendental type. If \ref{prop:equiv-VE:W} holds, then by the previous part of the proof also \ref{prop:equiv-VE:equiv} holds; thus, by Lemma \ref{bothtrans bothalgebr} both $E$ and $F$ are of transcendental type, and \ref{prop:equiv-VE:V} follows from Theorem \ref{teor:rank-VE}\ref{teor:rank-VE:trasc}.

If \ref{prop:equiv-VE:equiv} holds, then again $F$ is of transcendental type, and the fact that \ref{prop:equiv-VE:w} holds is exactly \cite[Satz 3.10]{KarpWah} (though note the slight difference in the definition -- see Remark \ref{mistake definition equivalence}); we give here  a proof for the sake of the reader. Without loss of generality, suppose that $K$ is algebraically closed. In order to show that $w_E(\phi)=w_F(\phi)$ for all $\phi\in K(X)$, it is sufficient to show that $w_E(X-\alpha)=w_F(X-\alpha)$ for every $\alpha\in K$. We have
\begin{align*}
w_E(X-\alpha)=\lim_{n\to\infty}v(s_n-\alpha)=v(s_{n}-\alpha),\;\;\forall n\geq n_1\\
w_F(X-\alpha)=\lim_{n\to\infty}v(t_n-\alpha)=v(t_{n}-\alpha),\;\;\forall n\geq m_1
\end{align*}
for some $n_1,m_1\in\N$, since both quantities are eventually constant. We also have that $w_E(X-\alpha)$ and $w_F(X-\alpha)$ are both strictly less $\delta$, since $\alpha\in K$ cannot be a pseudo-limit of $E$ and $F$, respectively. Hence, there exists $k_1\in\N$ such that for all $k>k_1$ we have $\delta_k>w_E(X-\alpha)$ and $\delta_k'>w_F(X-\alpha)$. Let $k>\max\{k_1,n_1,m_1\}$. There exists $k_2\geq k$ such that $\delta_k'<\delta_{k_2}$. Also, there exist $i_0,j_0\in\N$ such that for each $i\geq i_0$ and $j\geq j_0$ we have $v(s_i-t_j)>\delta_k'$. We have
$$v(s_{k_2}-\alpha)=v(s_{k_2}-s_m+s_m-t_m+t_m-\alpha)$$
Choose $m>\max\{k_2,i_0,j_0\}$. Then $v(s_{k_2}-s_m)$ and $v(s_m-t_m)$ are both strictly bigger than $\delta_k'>v(t_m-\alpha)$. Hence, $w_E(X-\alpha)=v(s_{k_2}-\alpha)=v(t_m-\alpha)=w_F(X-\alpha)$, and the claim is proved.

Suppose now that $E$ is of algebraic type. If \ref{prop:equiv-VE:equiv} holds, then by Lemma \ref{bothtrans bothalgebr} also $F$ is of algebraic type, and $E$ and $F$ have the same pseudo-limits with respect to any extension $u$ of $v$ to $\overline{K}$; hence, \ref{prop:equiv-VE:equiv} $\Longrightarrow$ \ref{prop:equiv-VE:limall}. Furthermore, \ref{prop:equiv-VE:limall} $\Longrightarrow$ \ref{prop:equiv-VE:limun} is obvious.

We now show that \ref{prop:equiv-VE:limun} implies \ref{prop:equiv-VE:V}. Let $\phi\in V_E$. By Proposition \ref{prop:qsn}, we have $v(\phi(s_n))=\lambda\delta_n+\gamma$, where $\lambda=\degdom_E\phi$ and $\gamma=u\left(\frac{\phi}{\phi_S}(\beta)\right)\in\Gamma_v$, where $\beta$ is a pseudo-limit of $E$ (and $\phi_S$ is defined as in the proposition). Similarly, $v(\phi(t_n))=\lambda'\delta'_n+\gamma'$; however, since $\limiti_E^v=\limiti_F^u$, it follows that $\lambda=\lambda'$ and $\gamma=\gamma'$. Furthermore, by Lemma \ref{lemma:kaplansky}, $\delta_E=\delta_F$, and thus $v(\phi(s_n))$ and $v(\phi(t_n))$ have the same limit $L$ as $n\to\infty$. Since $\phi\in V_E$, we have $v(\phi(s_n))\geq 0$ for large $n$, and so $L\geq 0$. If $L>0$, then also $v(\phi(t_n))>0$ for large $n$; this implies that $\phi\in V_F$. If $L=0$, then $\lambda\leq 0$; in particular, it must be $v(\phi(t_n))\geq 0$ for large $n$. Again, it follows that $\phi\in V_F$; therefore, $V_E\subseteq V_F$. Symmetrically, $V_F\subseteq V_E$, and thus $V_E=V_F$, as claimed.
\end{proof}

\subsection{A geometric interpretation of equivalence}\label{sect:geomint}
In this section, we give a geometric interpretation of Theorem \ref{teor:equiv-VE}. Let $\insiemeVEalg$ be the set of the valuation domains $V_E$, where $E\subset K$ is a pseudo-convergent sequence of algebraic type. Fix an extension $u$ of $v$ to the algebraic closure $\overline{K}$ of $K$. Then, to every valuation ring $V_E\in\insiemeVEalg$ is uniquely associated its set of pseudo-limits $\limiti_E^u\subset\overline{K}$; furthermore, since $\limiti_E^u=\beta_E+\Br_u(E)$, where $\beta_E\in\overline{K}$ is a pseudo-limit of $E$ with respect to $u$ (see Lemma \ref{lemma:kaplansky}), there is a well-defined and injective map
\begin{equation}\label{Sigma}
\begin{aligned}
\Sigma\colon\insiemeVEalg & \longrightarrow \cball_u(\overline{K})\\
V_E & \longmapsto \limiti_E^u,
\end{aligned}
\end{equation}
where $\cball_u(\overline{K})$ is the set of closed balls of the ultrametric space $\overline{K}$, endowed with the metric induced by $u$.

In general, $\Sigma$ is not surjective; to find its range, we introduce the following definition. For any $\beta\in\overline{K}$, we consider the minimum distance of the elements of $K$ from $\beta$, namely:
\begin{equation*}
d_u(\beta,K)=\inf\{d_u(\beta,x)=e^{-u(\beta-x)}\mid x\in K\}.
\end{equation*}
Note that $d_u(\beta,K)$ may be 0 even if $\beta\notin K$: this happens if and only if $\beta$ is in the completion of $K$ under $v$. If $V$ is a DVR, then the only closed balls of center $\beta\in\overline{K}$ which can arise as the set of pseudo-limits of a pseudo-convergent sequence $E\subset K$, are those of radius $0$ and with $\beta\in\widehat{K}$. If $V$ is non-discrete, we have the following result.

\begin{prop}\label{prop:immSigma}
Let $V$ be a non-discrete rank one valuation domain. Let $\beta\in\overline{K}$, $r\inR^+$ and $u$ an extension of $v$ to $\overline{K}$; let $B$ be the closed ball of center $\beta$ and radius $r$ with respect to $u$. Then, $B=\limiti_E^u$ for some pseudo-convergent sequence $E\subset K$ if and only if $r\geq d_u(\beta,K)$.
\end{prop}
\begin{proof}
Suppose $B=\limiti_E^u$, and let $E=\{s_n\}_{n\inN}$. Then, $\{d_u(\beta,s_n)\}_{n\inN}$ is a decreasing sequence of real numbers with limit $e^{-\delta}=r$, where $\delta$ is the breadth of $E$. By definition,
\begin{equation*}
d_u(\beta,K)=\inf\{d_u(\beta,s)\mid s\in K\}\leq d_u(\beta,s_n)
\end{equation*}
for every $n$, and thus $d_u(\beta,K)\leq r$.

Conversely, suppose $r\geq d_u(\beta,K)$. If $r=d(\beta,x)$ for some $x\in K$, take a sequence $Z=\{z_k\}_{k\inN}\subseteq K$ such that $v(z_k)$ is increasing and has limit $\delta=-\log(r)$. Then, $x+Z=\{x+z_k\}_{k\inN}$ is a pseudo-convergent sequence whose set of limits (in ($\overline{K},u)$) is $B$.

If $r\neq d(\beta,x)$ for every $x\in K$, we can take a sequence $E=\{s_n\}_{n\inN}$ such that $d_u(\beta,s_n)=r_n$ decreases to $r$. Then, $E$ is a pseudo-convergent sequence, and $\limiti_E^u=B$.
\end{proof}

If $V$ is not discrete, the next corollary gives a necessary and sufficient condition in order for the map defined in (\ref{Sigma}) to be  surjective. 

\begin{cor}
Suppose $V$ is not discrete. Then the map $\Sigma$ defined in (\ref{Sigma}) is surjective if and only if $\widehat{K}$ is algebraically closed.
\end{cor}
\begin{proof}
By Proposition \ref{prop:immSigma}, $\Sigma$ is surjective if and only if $\widehat{K}$ contains an algebraic closure of $K$. This happens if and only if $\widehat{K}$ is algebraically closed.
\end{proof}

\subsection{Extension of an Ostrowski valuation}
Let $w_E$ be the Ostrowski valuation on $K(X)$ associated to a pseudo-convergent sequence $E\subset K$, and let $u$ be an extension of $v$ to $\overline{K}$. The \emph{extension of $w_E$ to $\overline{K}(X)$ along $u$} is the valuation $\overline{w}_E$ defined by
\begin{equation*}
\overline{w}_E(\psi)=\lim_{n\to\infty} u(\psi(s_n))
\end{equation*}
for every $\psi\in\overline{K}(X)$. Clearly, $\overline{w}_E$ extends $w_E$ and has rank 1. A consequence of Theorem \ref{teor:equiv-VE} is that, if $E$ and $F$ are two pseudo-convergent sequences in $K$, the equality $w_E=w_F$ implies $\overline w_E=\overline w_F$, since these equalities are both equivalent to the fact that $E$ and $F$ are equivalent pseudo-convergent sequences (which does not depend on the field containing $E$ and $F$).

We show now that any extension of an Ostrowski valuation on $K(X)$ to $\overline K(X)$ is of this kind.

\begin{Thm}\label{Teor:extension wE}
Let $E=\{s_n\}_{n\in\N}\subset K$ be a pseudo-convergent sequence such that the associated map $w_E$ is a valuation. If $\overline{w}$ is an extension of $w_E$ to $\overline{K}(X)$ and $u$ is the restriction of $\overline{w}$ to $\overline{K}$, then $\overline{w}$ is equivalent to the extension $\overline{w}_E$ of $w_E$ to $\overline{K}(X)$ along $u$ (or, equivalently, the valuation domain of $\overline{w}$ is equal to $\overline{W}_E$).
\end{Thm}
\begin{proof}
The valuation $\overline{w}_E$ restricts of $w_E$ on $K(X)$ and to $u$ on $\overline{K}$. Suppose there is another valuation $w'$ on $\overline K(X)$ with these properties: then, by \cite[Chapt. VI, \S 8, 6., Corollaire 1]{Bourb}, there is a $K(X)$-automorphism $\sigma$ of $\overline{K}(X)$ such that $w'\circ\sigma$ is equivalent to $\overline{w}_E$, that is, $\rho(\overline{W}_E)=W'$, where $\rho=\sigma^{-1}$ and $W'$ is the valuation ring of $w'$.

Now
\begin{equation*}
\begin{aligned}
\rho(\overline{W}_E)= & \{\rho\circ\phi\in \overline{K}(X)\mid \lim_nu(\phi(s_n))\geq 0\}=\\
& \{\rho\circ\phi\in \overline{K}(X)\mid \lim_nu\circ\sigma\circ\rho(\phi(s_n))\geq 0\}.
\end{aligned}
\end{equation*}
Since $s_n\in K$ and $\rho|_K$ is the identity, $\rho(\phi(s_n))=(\rho\circ\phi)(s_n)$; hence,
\begin{equation*}
\begin{aligned}
\rho(\overline{W}_E)= & \{\rho\circ\phi\in \overline{K}(X)\mid \lim_n(u\circ\sigma)((\rho\circ\phi)(s_n))\geq 0\}=\\
& \{\psi\in \overline{K}(X)\mid \lim_n(u\circ\sigma)(\psi(s_n))\geq 0\}.
\end{aligned}
\end{equation*}

Since both $\overline{W}_E$ and $W'=\rho(\overline{W}_E)$ are extensions of $U$, the valuation domain of $u$, we have $u(t)=(u\circ\sigma)(t)$ for every $t\in\overline{K}$; in particular, this happens for $t=\psi(s_n)$. It follows that $\rho(\overline{W}_E)=W'=\overline{W}_E$, as claimed.
\end{proof}

\begin{oss}
We note that it is possible for two valuations $w_1,w_2$ of $\overline{K}(X)$ to be different even if their restriction to $K(X)$ and $\overline{K}$ are equal. For example, let $v$ be a valuation on $K$, and let $w$ be an extension of $v$ to $K(X)$. If $K$ is complete under the topology induced by $v$, then there exists a unique extension of $v$ to $\overline{K}$; on the other hand, $w$ can have more than one extension to $\overline{K}(X)$.

For an explicit example, suppose that $K$ is complete under $v$, let $\overline{v}$ be the unique extension of $v$ to $\overline{K}$ and let $\overline{V}$ be the valuation domain of $\overline{K}$ associated to $\overline{v}$. Let $\alpha,\beta\in\overline{K}$  be two distinct elements which are conjugate over $K$, and let $w$ be the valuation associated to the valuation domain
\begin{equation*}
W=\{\phi\in K(X)\mid \phi(\alpha)\in \overline V\}=\{\phi\in K(X)\mid \phi(\beta)\in \overline V\};
\end{equation*}
note that the second equality follows from the fact that $\alpha$ and $\beta$ are conjugate over $K$ (see also \cite[Theorem 3.2]{PerTransc}, where such valuation domains are studied; note that they belong  to the same class  of the valuation domains considered in Remark \ref{remark:BrE=0}). Then, $W$ extends to the following valuation rings  of $\overline{K}(X)$:
\begin{equation*}
\overline{W}_\alpha=\{\psi\in \overline{K}(X)\mid \psi(\alpha)\in \overline{V}\},\;\;\overline{W}_\beta=\{\psi\in \overline{K}(X)\mid \psi(\beta)\in \overline{V}\}.
\end{equation*}
%
However, $\overline{W}_\alpha\neq\overline{W}_\beta$: for example, if $t\in\overline{K}$ satisfies $\overline v(t)>\overline v(\beta-\alpha)$, then $f(X):=\inv{t}(X-\alpha)$ belongs to $\overline{W}_\alpha$ but not to $\overline{W}_\beta$ (again, the same conclusion follows from the aforementioned result \cite[Theorem 3.2]{PerTransc}).
\end{oss}

By means of Theorem \ref{Teor:extension wE}, in the next result without loss of generality we assume that the extension of $w_E$ to $\overline{K}(X)$ is equal to $\overline{w}_E$ (for some extension $u$ of $v$ to $\overline{K}$; clearly, $u=(\overline{w}_E)_{\mid\overline{K}}$).

The following is a variant of Theorem \ref{teor:linear}.
\begin{prop}\label{prop:linear wE}
Let $\phi\in K(X)$ and let $E=\{s_n\}_{n\in\N}\subset K$ be a pseudo-convergent sequence. Let $\overline w_E$ be an extension of $w_E$ to $\overline{K}(X)$, and let $\theta_1,\theta_2\inR$ be  such that $C=\{t\in \overline K\mid\theta_1<\overline w_E(X-t)<\theta_2\}$ does not contain any critical point of $\phi$. Then, there are $\lambda\inZ$, $\gamma\in\insQ\Gamma_v$ such that
\begin{equation*}
v(\phi(t))=\lambda w_E(X-t)+\gamma
\end{equation*}
for every $t\in K\cap C$. More precisely, if $S$ is the multiset of critical points $\alpha$ of $\phi$ such that $\overline w_E(X-\alpha)\geq\theta_2$, then $\lambda$ is the weighted sum of $S$ and $\gamma=\overline w_E\left(\frac{\phi}{\phi_{S}}\right)$.
\end{prop}
\begin{proof}
Let $\phi(X)=c\prod_{\alpha\in S}(X-\alpha)^{\epsilon_\alpha}\prod_{\beta\in S'}(X-\beta)^{\epsilon_\beta}$, where $S'$ is the multiset of critical points of $\phi$ with $\overline{w}_E<\theta_2$. Let $t\in K\cap C$ and let $u$ be the restriction of $\overline{w}_E$ to $\overline{K}$. As in the proof of Theorem \ref{teor:linear}, writing $u(t-\alpha)=\overline w_E(t-\alpha)=\overline w_E(t-X+X-\alpha)$ we see that $u(t-\alpha)=w_E(X-t)$ if $\overline w_E(X-\alpha)\geq\theta_2$, while $u(t-\alpha)=\overline w_E(X-\alpha)$ if $\overline w_E(X-\alpha)\leq\theta_1$ (note that by assumption on $C$ there is no critical point $\alpha$ of $\phi$ such that $\theta_1<\overline w_E(X-\alpha)<\theta_2$). Hence,
\begin{equation*}
v(\phi(t))=v(c)+\sum_{\alpha\in S}\epsilon_\alpha w_E(X-t)+\sum_{\beta\in S'}\epsilon_\beta \overline w_E(X-\beta) =\lambda w_E(X-t)+\gamma,
\end{equation*}
as claimed.
\end{proof}

Given a pseudo-convergent sequence $E\subset K$, an extension $\overline{w}_E$ of $w_E$ and a rational function $\phi\in K(X)$, we define
$$\delta_{\phi,E}=\max\{\overline w_E(X-\alpha)\mid \alpha\text{~is a critical point of~}\phi\}$$
(which we simply write $\delta_\phi$ if $E$ is understood from the context). By Remark \ref{Rem:osservazioni su wE}\ref{Rem:osservazioni su wE:X},  $\delta_{\phi}\leq\delta$, and $\delta_\phi<\delta$ if no critical point of $\phi$ is a pseudo-limit of $E$; in particular, this happens if $E$ is of transcendental type.

\begin{Cor}\label{cor:constant value phi}
Let $E\subset K$ be a pseudo-convergent sequence and let $\phi\in K(X)$, and suppose that none of the  critical points of $\phi$ is a pseudo-limit of $E$ (with respect to $u=(\overline{w}_E)_{\mid\overline{K}}$). Then:
\begin{enumerate}[(a)]
\item\label{cor:constant value phi:a} if $\delta_{\phi}<w_E(X-t)\leq\delta_E$, then $v(\phi(t))=w_E(\phi)$;
\item\label{cor:constant value phi:b} if $E$ is of algebraic type and $\alpha\in\mathcal{L}_E^u$, then $w_E(\phi)=u(\phi(\alpha))$.
\end{enumerate}
\end{Cor}
\begin{proof}
Let $E=\{s_n\}_{n\in\N}$.  Since no critical point $\beta$ of $\phi$ satisfies $\overline{w}_E(X-\beta)\geq\delta_\phi$, by Proposition \ref{prop:linear wE} we have $v(\phi(t))=\overline{w}_E(\phi)=w_E(\phi)$ for every $t\in C=\{s\in\overline{K} \mid \delta_{\phi}<\overline{w}_E(X-s)\leq\delta_E\}$, so claim \ref{cor:constant value phi:a} is proved. Claim \ref{cor:constant value phi:b} follows by Proposition \ref{prop:qsn}\ref{prop:qsn:alg}.
\end{proof}

\section{Spaces of valuation domains associated to pseudo-convergent sequences}\label{sect:spaces}
We are now interested in studying, from a topological point of view, the sets formed by the valuation rings $V_E$ and $W_E$ induced by the pseudo-convergent sequences $E$ in $K$; again we are still assuming that $V$ is a rank one valuation domain. The topologies we are interested in are the Zariski and the constructible topologies (see Section \ref{sect:background} for the definitions). Since we are mainly interested in the former, unless stated otherwise, all the spaces are endowed with the Zariski topology.

We set:
\begin{equation*}
\insiemeVE=\{V_E\mid E\subset K \text{~is a pseudo-convergent sequence}\}
\end{equation*}
and
\begin{equation*}
\mathcal{W}=\{W_E\mid E\subset K \text{~is a pseudo-convergent sequence  and~} w_E\text{~is  a valuation}\}.
\end{equation*}
By the results of Section \ref{sect:WE}, the elements of $\mathcal{W}$ are the rings $W_E$, when $E\subset K$ is a pseudo-convergent sequence which is either of transcendental type or of algebraic type and non-zero breadth ideal. 

When $V$ is discrete, we have the following result.
\begin{teor}\label{teor:Vdisc}
\cite[Theorem 3.4]{PerTransc} Let $V$ be a DVR. Then, $\mathcal{V}$ is homeomorphic to $\widehat{K}$.
\end{teor}
The homeomorphism can also be described explicitly: indeed, if $V$ is a DVR then $\mathcal{V}$ contains only the rings of the form $W_\alpha$ (see Remark \ref{remark:BrE=0}) and we just send $W_\alpha$ to $\alpha$. Furthermore, in this context, $\mathcal{W}$ is a subset of $\mathcal{V}$, and corresponds to the elements of $\widehat{K}$ that are transcendental over $V$. In view of these facts, we are mainly interested in the case when $V$ is not discrete.

\begin{oss}
If $V$ is discrete and $\widehat{K}$ is algebraic over $K$ (for example, if $K$ is complete; however, this is not a necessary condition, see \cite[\textsection 1, p. 394]{ribenboim-art}) then $\widehat{K}$ contains no elements that are transcendental over $K$, and thus $\mathcal{W}$ is empty. Conversely, if $\mathcal{W}$ is empty then $V$ must be discrete, otherwise we have rings $W_E$ coming from non-Cauchy pseudo-convergent sequences $E$, and $\widehat{K}$ must be algebraic over $K$.
\end{oss}

From now on, we assume that $\mathcal{W}$ is nonempty.

We start by studying $\mathcal{W}$: indeed, the fact that every Ostrowski valutation domain $W_E$ has rank one has strong consequences on the topology of $\mathcal{W}$. Recall that a topological space $X$  is said to be  \emph{zero-dimensional} if it is $T_1$ and each point $x\in X$ has a neighborhood base consisting of open-closed sets, or, equivalently, if, for each $x\in X$ and closed set $C\subset X$, there exists an open-closed set containing $x$ and not meeting $C$  \cite[f-6]{encytop}.

\begin{prop}\label{prop:spazioW}
The Zariski and the constructible topologies agree on $\mathcal{W}$. In particular, $\mathcal{W}$ is a zero-dimensional space.
\end{prop}
\begin{proof}
The intersection of the maximal ideals of the elements of $\mathcal{W}$ contains the maximal ideal $M$ of $V$, and thus it is nonzero. Since every $W_E$ has rank 1, the claims follow by \cite[Proposition 2.4(b)]{olberding-PIT} and the definition of zero-dimensional space.
\end{proof}

Let $S\subseteq K$ be a subset. The ring of \emph{integer-valued rational functions} on $S$ is the ring
$$\Int^R(S,V)=\{\phi\in K(X) \mid \phi(S)\subseteq V\}$$
(see \cite[Chapter X]{CaCh} for a general reference). Note that $\Int^R(S,V)$ may be equal only to $V$ (for example if $S=K$ and $K$ is algebraically closed, see \cite[Proposition 2.4]{CahLop}).

\begin{prop}\label{prop:Wcompatto}
The space $\mathcal{W}$ is not compact.
\end{prop}
\begin{proof}
We claim that
\begin{equation*}
\bigcap_{\substack{E\subset K\\E\textnormal{ pseudo-conv.}}}W_E=\Int^R(K,V),
\end{equation*}
Let $\phi\in \Int^R(K,V)$. Then, clearly $\phi\in V_E\subseteq W_E$ for all pseudo-convergent sequences $E$, by definition of $V_E$.  Conversely, if $\phi(K)\nsubseteq V$, then there is a $t\in K$ such that $\phi(t)\notin V$; since $V$ is closed in $K$ and a rational function induces a continuous function (from the subset of $K$ on which it is defined to $K$), there is a ball $B(t,r)$ such that $\phi(s)\notin V$ for all $s\in B(t,r)$. Choose $s\in B(t,r)$ which is not a critical point of $\phi$ and let $E=\{s_n\}_{n\inN}\subset K$ be a pseudo-convergent Cauchy sequence  with limit $s$. Then $w_E(\phi)=\lim v(\phi(s_n))=v(\phi(s))<0$, that is, $\phi\notin W_E$.

The intersection of the maximal ideals of the Ostrowski valuation overrings is $M\neq(0)$; hence, if $\mathcal{W}$ is compact then by \cite[Theorem 5.3]{olberding-PIT} $\Int^R(K,V)$ is a one-dimensional Pr\"ufer domain with quotient field $K(X)$. Hence all rings between $\Int^R(K,V)$ and $K(X)$ have dimension (at most) $1$. However, if $E$ is a pseudo-convergent Cauchy sequence with limit in $K$, then $V_E$ has dimension $2$ (by Theorem \ref{teor:rank-VE}\ref{teor:rank-VE:alg-inf}) but $\Int^R(K,V)\subset V_E$, a contradiction. Therefore, $\mathcal{W}$ is not compact, as claimed.
\end{proof}

\begin{oss}
When $V$ is discrete and countable, the space $\mathcal{W}$ is not even locally compact. Indeed, as $V$ is discrete, $\mathcal{W}$ is a subset of $\insiemeVE$, and in the homeomorphism of $\mathcal V$ into $\widehat{K}$ (Theorem \ref{teor:Vdisc}) $\mathcal W$ corresponds to the subset $X$ of those elements which are transcendental over $K$.

Furthermore, when $V$ is countable, also the algebraic closure is countable, while the completion $\overline{K}$ and all its open subsets are uncountable. Hence, every open ball contains elements that are transcendental over $K$, and thus $X$ is dense.

If $X$ were locally compact, then by \cite[Theorem 18.4]{willard} it should be an intersection of an open set $O$ and a closed set $C$ of $\widehat{K}$ (since $\widehat{K}$ is Hausdorff). Since $X$ is dense, we should have $C=\widehat{K}$, and thus $X=O$ should be open. However, also $\widehat{K}\setminus X$ is dense (since all elements of $\widehat{K}$ are limits of sequences in $K$) and thus $X$ cannot be open. Therefore, $\mathcal{W}$ is not locally compact.

We conjecture that $\mathcal{W}$, if nonempty, is never locally compact.
\end{oss}

In order to study more closely the Zariski topology, we now want to study convergence of sequences of valuation domains. To do so, we give in the next two lemmas two criteria to establish when $V_E$ belongs to $B(\phi)$. We introduce the following  notation: if $\beta\in K$, $\gamma_1\in\Gamma_v$ and $\gamma_2\in \Gamma_v\cup\{\infty\}$ with $\gamma_1<\gamma_2$, the \emph{annulus} of center $\beta$ and radii $\gamma_1$ and $\gamma_2$ is
\begin{equation*}
\corona_v(\beta,\gamma_1,\gamma_2)=\{s\in K\mid \gamma_1<v(\beta-s)<\gamma_2\}.
\end{equation*}
Note that this definition is a special case of the definition given in (\ref{CbetaDelta}), when $V$ has rank one. When the valuation $v$ is understood from the context, we shall write simply $\corona(\beta,\gamma_1,\gamma_2)$ for $\corona_v(\beta,\gamma_1,\gamma_2)$. 


\begin{lemma}\label{lemma:corona}
Let $E\subset K$ be a pseudo-convergent sequence of algebraic type with breadth $\delta$, let $\beta\in\limiti_E^u$ and let $\phi\in K(X)$. The following are equivalent:
\begin{enumerate}[(i)]
\item\label{lemma:corona:phi} $\phi\in V_E$;
\item\label{lemma:corona:QGamma} there are $\theta_1\in\insQ\Gamma_v$, $\theta_2\in\insQ\Gamma_v\cup\{\infty\}$ such that $\theta_1<\delta\leq\theta_2$ and such that $\phi(s)\in V$ for all $s\in\corona_u(\beta,\theta_1,\theta_2)$;
\item\label{lemma:corona:tau} there is $\tau\in\Gamma_v$, $0<\tau<\delta$ such that $\phi(s)\in V$ for all $s\in\corona_u(\beta,\tau,\delta)$.
\end{enumerate}
\end{lemma}
\begin{proof}
\ref{lemma:corona:phi} $\Longrightarrow$ \ref{lemma:corona:QGamma} Let $\zeta_1<\zeta_2$ be two elements in $\insQ\Gamma_v$ such that $\zeta_1<\delta\leq\zeta_2$ and there is no critical point of $\phi$ in $C=\corona_u(\beta,\zeta_1,\zeta_2)$. By Theorem \ref{teor:linear}, there are $\lambda\inZ$, $\gamma\in\Gamma_v$ such that $v(\phi(s))=\lambda u(\beta-s)+\gamma$ for all $s\in C$. Let $I=\{h\in(\zeta_1,\zeta_2)\mid \lambda h+\gamma\geq 0\}$; then, $I$ is an interval with endpoints $\theta_1,\theta_2\in\insQ\Gamma_v$, and $\phi(s)\in V$ for all $s\in C'=\corona_u(\beta,\theta_1,\theta_2)$; we need only to show that $\theta_1<\delta\leq\theta_2$.

Since $\phi\in V_E$, and $s_n\in C'$ for large $n$, we have $\lambda\delta_n+\gamma\geq 0$ for all $n$ large enough, where $\{\delta_n\}_{n\inN}$ is the gauge of $E$; since $\delta_n\nearrow\delta$ and $\zeta_1<\delta$, it follows that there is an interval $(\tau,\delta)\subseteq I$, and so $\theta_1<\delta\leq\theta_2$. The claim is proved.

\ref{lemma:corona:QGamma} $\Longrightarrow$ \ref{lemma:corona:tau} is obvious.

\ref{lemma:corona:tau} $\Longrightarrow$ \ref{lemma:corona:phi} Suppose that there is an annulus $C=\corona_u(\beta,\tau,\delta)$ with this property. Since $\delta$ is the breadth of $E$, for large $n$ we have $s_n\in C$; hence, $\phi(s_n)\in V$ and thus $\phi\in V_E$.
\end{proof}

\begin{oss}\label{oss:corona-contrario}
The exact same proof of the previous proposition can be used to show a converse: $\phi\notin V_E$ if and only if there is an annulus $\corona=\corona(\beta,\tau,\delta)$ such that $\phi(t)\notin V$ for all $t\in\corona$ (and similarly for the version with $\theta_1$ and $\theta_2$).
\end{oss}

The following technical lemma is based on the Ostrowski valuation $\overline{w}_E$. Note that the set $C$ defined below is essentially an annulus with respect to $\overline{w}_E$.
\begin{lemma}\label{lemma:corona-phiVEVF}
Let $E=\{s_n\}_{n\inN}\subset K$ be a pseudo-convergent sequence, and let $\phi\in K(X)$; let $u$ be an extension of $v$ to $\overline{K}$ and let $\overline w_E$ be the extension of $w_E$ to $\overline{K}(X)$ along $u$. There is a $\delta'<\delta_E$ such that, given $C=\{s\in\overline{K}\mid \delta'<\overline{w}_E(X-s)<\delta_E\}$, whenever $F$ is a pseudo-convergent sequence such that $\delta_F\geq\delta_E$ and $\limiti_F\cap C\neq\emptyset$, we have $\phi\in V_F$ if and only if $\phi\in V_E$.
\end{lemma}
\begin{proof}
Let $S\subset\overline K$ be the multiset of critical points of $\phi$, and let $\delta_1=\sup\{\overline{w}_E(X-\alpha)\mid \alpha\in S,~\overline{w}_E(X-\alpha)<\delta_E\}$. Then, there are no critical points of $\phi$ in $C_1=\{s\in\overline{K}\mid \delta_1<\overline{w}_E(X-s)<\delta_E\}$; by Proposition \ref{prop:linear wE}, there are $\lambda\inZ$, $\gamma\in\insQ\Gamma_v$ such that $u(\phi(t))=\lambda\overline{w}_E(X-t)+\gamma$ for every $t\in C_1$.  Since $v(\phi(s_n))\to w_E(\phi)$ and $s_n$ is eventually in $C_1$, we can find $\delta'\in[\delta_1,\delta_E)$ such that the quantity $u(\phi(t))$ is either positive, negative or zero for all $t\in C=\{s\in\overline{K}\mid \delta'<\overline{w}_E(X-s)<\delta_E\}$.

Suppose now $F=\{t_m\}_{m\inN}\subset K$ is a pseudo-convergent sequence with breadth $\delta_F\geq\delta_E$ and such that $\limiti_F\cap C\neq\emptyset$: then, if $t\in\limiti_F\cap C$, we have $\overline{w}_E(X-t_m)=\overline{w}_E(X-t+t-t_m)=\overline{w}_E(X-t)$, for all $m\in\N$ sufficiently large, since $\overline{w}_E(X-t)<\delta_E$ and $v(t_m-t)\nearrow\delta_F$ which is greater than or equal to $\delta$ (and so, it is eventually greater than $\overline{w}_E(X-t)$). Hence, $t_m$ is eventually in $C$ and thus $v(\phi(t_m))$ is eventually nonnegative if so is $v(\phi(s_n))$ (in which case $\phi\in V_E\cap V_F$), while it is eventually negative if $v(\phi(s_n))$ is eventually negative (and so $\phi\notin V_E$ and $\phi\notin V_F$). The claim is proved.
\end{proof}
We note that, when we are in the hypothesis of Corollary \ref{cor:constant value phi} (that is, if $\phi$ has no critical point which is a pseudo-limit of $E$), the value $\delta'$ of the previous proposition can be taken to be equal to $\delta_{\phi,E}$.

As a first step in the study of $\insiemeVE$, we analyze the convergence of sequences in $\Zar(K(X)|V)^\cons$.
\begin{prop}\label{prop:limitcostr}
Let $E=\{s_n\}_{n\inN}\subset K$ be a pseudo-convergent sequence of algebraic type with breadth $\delta$, and, for each $n\in\N$, let $\zeta_n\in [\delta,\infty]$. For each $n\in\N$, let $E_n\subset K$ is a pseudo-convergent sequence with pseudo-limit $s_n$ and breadth $\zeta_n$. Then:
\begin{enumerate}[(a)]
\item $V_E$ is a limit of $\{V_{E_n}\}_{n\in\N}$ in $\Zar(K(X)|V)^\cons$;
\item if $\zeta_n\neq\infty$ for every $n$, then $V_E$ is a limit of $\{W_{E_n}\}_{n\in\N}$ in $\Zar(K(X)|V)^\cons$.
\end{enumerate}
\end{prop}
Note that, if $V_E$ is a limit in the constructible topology, it is a limit also in the Zariski topology.
\begin{proof}
Let $\mathcal{X}=\Zar(K(X)|V)^\cons$; we need to show that, if $V_E\in\Omega$ for some open set $\Omega$, then $V_{E_n},W_{E_n}\in\Omega$ for large $n$; without loss of generality, we can consider only the cases $\Omega=B(\phi)$ and $\Omega=\mathcal{X}\setminus B(\phi)$, where $\phi\in K(X)$. This amounts to prove that $V_E\in B(\phi)$ if and only if $V_{E_n}\in B(\phi)$ (respectively, $W_{E_n}\in B(\phi)$) for all large $n$.

Suppose first that $E$ has a pseudo-limit $s\in K$. By Proposition \ref{lemma:corona}, there is an annulus $\corona=\corona(s,\tau,\delta)$ such that $\phi(t)\in V$ for all $t\in\corona$. There is a $N$ such that $s_n\in\corona$ for $n\geq N$; hence, for these $n$, $\limiti_{E_n}\cap\corona\neq\emptyset$. For all $t\in\corona$, we have $w_E(X-t)=u(s-t)$; hence, $\corona=\{t\in K\mid \tau<w_E(X-t)<\delta\}$. Therefore, we can apply Lemma \ref{lemma:corona-phiVEVF}, and so $V_E\in B(\phi)$ if and only if $V_{E_n}\in B(\phi)$ for $n\geq N$. Thus, the sequence $V_{E_n}$ tends to $V_E$ in the constructible topology.

Since $V_{E_n}\subseteq W_{E_n}$, we also have that if $V_E\in B(\phi)$ then $W_{E_n}\in B(\phi)$ for large $n$. Furthermore, without loss of generality, $\corona$ does not contain any critical point of $\phi$ and $v(\phi(t))=\lambda v(t-s)+\gamma$, for each $t\in \corona$, for some $\lambda\in\Z$ and $\gamma\in\Gamma_v$ by Theorem \ref{teor:linear}; since $v(t_m-s)=v(t-s)$, for all $m\in\N$ sufficiently large, where $t\in\mathcal{L}_F\cap\corona$, then $v(\phi(t_m))=v(\phi(t))$ for all such $m$, and so $w_F(\phi)=v(\phi(t))$: hence, if $V_E\notin B(\phi)$ then also $W_{E_n}\notin B(\phi)$. It follows that also the sequence $W_{E_n}$ tends to $V_E$ in $\Zar(K(X)|V)^\cons$.

Suppose now that $E$ has a limit $\beta\in\overline{K}$ with respect to some extension $u$ of $v$ to $\overline{K}$; let $U\subset\overline K$ be the valuation domain of $u$. By the previous part of the proof, $U_E$ is the limit of the sequence $U_{E_n}$ in $\Zar(\overline{K}(X)|U)^\cons$. The restriction map $\pi:\Zar(\overline{K}(X)|U)^\cons\longrightarrow\Zar(K(X)|V)^\cons$, $W\mapsto W\cap K(X)$, is continuous; hence, $\pi(U_{E_n})\to \pi(U_E)$. However, $\pi(U_{E_n})=V_{E_n}$ and $\pi(U_E)=V_E$; the claim is proved. The same reasoning applies to the sequence $\{W_{E_n}\}_{n\in\N}$.

The claim about the Zariski topology follows since the constructible topology is finer than the Zariski topology.
\end{proof}

\begin{Ex}\label{ex:seq-Vsn}
Let $E=\{s_n\}_{n\in\N}$ be a pseudo-convergent sequence of algebraic type and, for each $n\in\N$, let $W_{s_n}=\{\phi\in K(X)\mid \phi(s_n)\in V\}$. Then, by the previous lemma, $\{W_{s_n}\}_{n\in\N}$ converges to $V_E$ in the constructible and in the Zariski topology. 
\end{Ex}

Since we are working with the Zariski topology on $\insiemeVE$ and $\mathcal{W}$, for ease of notation we set
\begin{align*}
B^{\insiemeVE}(\phi)=\{V_E\in\insiemeVE\mid V_E\ni\phi\}=B(\phi)\cap\insiemeVE,\\
B^{\mathcal{W}}(\phi)=\{W_E\in\mathcal{W}\mid W_E\ni\phi\}=B(\phi)\cap\mathcal{W}.
\end{align*}

We denote by $\insiemeVE(\bullet,\delta)$ the set of valuation domains $V_E$ such that $E$ has breadth $\delta$.
\begin{prop}\label{prop:costr-VE}
Let $V$ be a valuation domain of rank 1 which is not discrete. The following are equivalent:
\begin{enumerate}[(i)]
\item\label{prop:costr-VE:res} the residue field of $V$ is finite;
\item\label{prop:costr-VE:tutto} the Zariski and the constructible topologies coincide on $\insiemeVE$;
\item\label{prop:costr-VE:leq} there is a $\delta\in\insR\cup\{+\infty\}$ such that the the Zariski and the constructible topologies coincide on $\bigcup_{\delta'\leq\delta}\insiemeVE(\bullet,\delta')$;
\item\label{prop:costr-VE:min} there is a $\delta\in\insR\cup\{+\infty\}$ such that the the Zariski and the constructible topologies coincide on $\bigcup_{\delta'<\delta}\insiemeVE(\bullet,\delta')$.
\end{enumerate}
\end{prop}
When $V$ is discrete, $\insiemeVE$ reduces to $\insiemeVE(\bullet,\infty)$.

\begin{proof}
\ref{prop:costr-VE:res} $\Longrightarrow$ \ref{prop:costr-VE:tutto} To show that the Zariski and the constructible topologies coincide, it is enough to show that $B(\phi)$ is closed in the Zariski topology for every $\phi\in K(X)$. Let thus $E=\{s_n\}_{n\inN}$ be a pseudo-convergent sequence with breadth $\delta$ such that $V_E\notin B(\phi)$; we want to show that there is an open neighborhood of $V_E$ disjoint from $B(\phi)$.

If $E$ is of transcendental type, then $V_E=W_E$; since the Zariski and the constructible topologies agree on $\mathcal{W}$ (Proposition \ref{prop:spazioW}), the set $B^{\mathcal{W}}(\phi)$ is closed in $\mathcal{W}$, and thus there are $\psi_1,\ldots,\psi_k$ such that $W_E\in B^{\mathcal{W}}(\psi_1,\ldots,\psi_k)$ but $B^{\mathcal{W}}(\psi_1,\ldots,\psi_k)\cap B^{\mathcal{W}}(\phi)=\emptyset$. In particular, $V_E\in B^{\mathcal{V}}(\psi_1,\ldots,\psi_k)$; on the other hand, if $V_F\in B^{\mathcal{V}}(\psi_1,\ldots,\psi_k)\cap B^{\mathcal{V}}(\phi)$, then $\psi_1,\ldots,\psi_k,\phi\in V_F\subseteq W_F$, and thus $W_F\in B^{\mathcal{W}}(\psi_1,\ldots,\psi_k)\cap B^{\mathcal{W}}(\phi)$, a contradiction.  Hence, $B^{\mathcal{V}}(\psi_1,\ldots,\psi_k)$ and $B^{\mathcal{V}}(\phi)$ are disjoint, and $B^{\mathcal{V}}(\psi_1,\ldots,\psi_k)$ is the required neighborhood.

Suppose $E$ is of algebraic type without pseudo-limits in $K$; let $\alpha\in\overline K\setminus K$ be a pseudo-limit of $E$ with respect to an extension $u$ of $v$ to $\overline{K}$. By Proposition \ref{lemma:corona} and Remark \ref{oss:corona-contrario}, there is an annulus $\corona=\corona_u(\alpha,\theta_1,\theta_2)$  with $\theta_1,\theta_2\in\insQ\Gamma_v$, $\theta_1<\delta\leq\theta_2$, such that $\phi(t)\notin V$ for all $t\in\corona$. Let $s\in\corona$; then $\theta_1<u(\alpha-s)<\delta$, because otherwise $s$ would be a pseudo-limit of $E$. Let $d\in K$ be such that $\theta_1<v(d)<u(\alpha-s)$. Then, $V_E\in B\left(\frac{X-s}{d}\right)$, since, for large $n$, $v(s_n-s)-v(d)=u(s_n-\alpha+\alpha-s)-v(d)=u(\alpha-s)-v(d)\geq 0$. On the other hand, if $t\in K$ is such that $v\left(\frac{t-s}{d}\right)\geq 0$, then $v(t-s)\geq v(d)>\theta_1$, so $u(t-\alpha)=u(t-s+s-\alpha)>\theta_1$. Since $u(t-\alpha)<\delta$ because $E$ has no pseudo-limits in $K$, it follows that $t\in\corona$, so that $\phi(t)\notin V$; in particular, $B\left(\frac{X-s}{d}\right)$ is a neighborhood of $V_E$ disjoint from $B(\phi)$.

Suppose now that $E$ is of algebraic type with a pseudo-limit $s\in K$. If $\delta\notin\insQ\Gamma_v$, then $V_E=W_E$ by Theorem \ref{teor:equiv-VE}, so  the claim follows as in the transcendental case. Suppose $\delta\in\insQ\Gamma_v$, and let $k\inN^+$ be such that $k\delta\in\Gamma_v$. By Proposition \ref{lemma:corona} and Remark \ref{oss:corona-contrario}, there is an annulus $\corona(s,\tau,\delta)$, with $\tau<\delta$, such that $\phi(t)\notin V$ for all $t\in\corona(s,\tau,\delta)$. Let $d\in K$ be an element such that $v(d)\in(\tau,\delta)$; then, $V_E\in B\left(\frac{X-s}{d}\right)$.

Let $u_1,\ldots,u_r$ be a complete set of representatives for the residue field of $V$; suppose that $u_1\in M$ and $u_i\in V\setminus M$ for $i=2,\ldots,r$. Let $z\in K$ be an element of valuation $k\delta$. Let
\begin{equation*}
\psi(X)=\frac{z^{r}}{((X-s)^k-zu_1)\cdots((X-s)^k-zu_r)};
\end{equation*}
we claim that $\psi(t)\in V$ if and only if $v(t-s)<\delta$. 

Indeed, if $v(t-s)<\delta$ then $v((t-s)^k)=kv(t-s)<k\delta\leq v(zu_i)$ for $i=1,\ldots,r$, and thus
\begin{equation*}
v(\psi(t))=rk\delta-rkv(t-s)>0.
\end{equation*}
If $v(t-s)>\delta$, then $v((t-s)^k-zu_i)= k\delta$ for $i=2,\ldots,r$ and $v((t-s)^k-zu_1)>k\delta$, and thus
\begin{equation*}
v(\psi(t))<rk\delta-rk\delta=0.
\end{equation*}
If $v(t-s)=\delta$, then $v((t-s)^k)=k\delta=v(z)$; since $u_1,\ldots,u_r$ are a complete set of representatives, there is a (unique) $i\in\{2,\ldots,r\}$ such that $v((t-s)^k-zu_i)>k\delta$, while $v((t-s)^k-zu_j)=k\delta$ for $j\in\{1,\ldots,r\}\setminus\{i\}$. Hence,
\begin{equation*}
v(\psi(t))=rk\delta-(r-1)k\delta-v((t-s)^k-zu_i)=k\delta-v((t-s)^k-zu_i)<0.
\end{equation*}

In particular, $V_E\in B(\psi)$ by Proposition \ref{lemma:corona}; furthermore, if $\psi(t),\frac{t-s}{d}\in V$, then $t\in\corona$. Hence, $B(\psi)\cap B\left(\frac{X-s}{d}\right)\cap B(\phi)=\emptyset$, and thus $B(\psi)\cap B\left(\frac{X-s}{d}\right)$ is a neighborhood of $V_E$ disjoint from $B(\phi)$. It follows that $B(\phi)$ is closed, as claimed.

\medskip

\ref{prop:costr-VE:tutto} $\Longrightarrow$ \ref{prop:costr-VE:leq} and \ref{prop:costr-VE:min} are obvious.

Suppose now that either \ref{prop:costr-VE:leq} or \ref{prop:costr-VE:min} hold for some $\delta$, and let $\mathcal X$ be $\bigcup_{\delta'\leq\delta}\insiemeVE(\bullet,\delta')$ or $\bigcup_{\delta'<\delta}\insiemeVE(\bullet,\delta')$, accordingly. Suppose that the residue field of $V$ is infinite. Let $c\in K$ such that $\eta=v(c)<\delta$: we claim that $B(c^{-1}X)\cap\mathcal{X}$ is not closed in $\mathcal X$.

Indeed, let $E=\{s_n\}_{n\inN}$ be a pseudo-convergent sequence with breadth $\eta$ and having $0$ as a pseudo-limit. Then, $V_E\notin B(c^{-1}X)$.  Suppose there is a neighborhood of $V_E$ disjoint from $B(c^{-1}X)$: then, there are $\psi_1,\ldots,\psi_k$ such that $V_E\in B(\psi_1,\ldots,\psi_k)$ and such that $B(\psi_1,\ldots,\psi_k)\cap B(c^{-1}X)=\emptyset$.

Fix an extension $u$ of $v$ to $\overline{K}$. Let $\beta_1,\ldots,\beta_m$ be the critical points of $\psi_1,\ldots,\psi_k$ having valuation $\eta$ under $u$ (if there are any). Since the residue field of $V$ is infinite, there is a $t\in K$ such that $v(t)=\eta$ and such that $u(t-\beta_i)=\eta$ for all $i$. We claim that $v(\psi_i(t))\geq 0$ for all $i$.

Indeed, fix $i$, and let $\alpha_1,\ldots,\alpha_r$ be the critical points of $\psi=\psi_i$. By construction, we have
\begin{equation}\label{u(t-alphajleqeta)}
u(t-\alpha_j)=\begin{cases}
u(\alpha_j) & \text{if~}u(\alpha_j)<\eta,\\
v(t)=\eta & \text{if~}u(\alpha_j)\geq\eta.
\end{cases}
\end{equation}
In particular, a direct calculation gives $v(\psi(t))=\lambda\eta+\gamma$, where $\lambda$ is the weighted sum of the critical points of $\psi$ in the closed ball $B(0,e^{-\eta})$ and $\gamma\in\Gamma_v$. By Corollary \ref{cor:wEphi}, it follows that $v(\psi(t))=w_E(\psi)$; in particular, $v(\psi(t))\geq 0$ since $\psi\in V_E$. Therefore, $v(\psi_i(t))\geq 0$ for all $i$. Furthermore, we claim that
\begin{equation}\label{claim}
v(\psi_i(t'))=v(\psi_i(t))\geq 0,\;\textnormal{ for all } t' \textnormal{ such that }v(t-t')>\eta
\end{equation}
In fact, by (\ref{u(t-alphajleqeta)}) we have $\eta\geq u(t-\alpha_i)$, so $u(t'-\alpha_i)=u(t-\alpha_i)$ for all $i=1,\ldots,r$ and the claim follows.

Hence, if $\eta'>\eta$ and $F=\{t_n\}_{n\inN}$ is a pseudo-convergent sequence of breadth $\eta'$ and pseudo-limit $t$, then $V_F\in B(\psi_1,\ldots,\psi_k)$ by (\ref{claim}), since $v(t-t_n)>\eta$ for large $n$. In particular, we must have $v(t)=v(t_n)$ for every $n$, since $v(t)=\eta<\eta'$ and $v(t-t_n)\nearrow\eta'$, so $0$ is not a pseudo-limit of $F$ (thus, $v(t_n)$ is eventually constant). Hence, $V_F$ also belongs to $B(c^{-1}X)$; therefore, if we choose $\eta'\in(\eta,\delta)$, we have $V_F\in B(\psi_1,\ldots,\psi_k)\cap B(c^{-1}X)\cap\mathcal X$, against our choice of $\psi_1,\ldots,\psi_k$. Therefore, $B(c^{-1}X)\cap \mathcal X$ is not closed, and the constructible topology does not agree with the Zariski topology. By contradiction, \ref{prop:costr-VE:res} holds.
\end{proof}


To conclude this section, we study the function from $\mathcal W$ to $\mathcal V$ which maps each $W_E$ to $V_E$. We need the following lemma.
\begin{lemma}\label{lemma:finiteness}
Let $\phi\in K(X)$ and $\delta\in\R$. Let $S$ be the set of valuation domains $V_F$, with $F=\{t_n\}_{n\inN}$, such that $v(\phi(t_n))\nearrow \delta$. Then, $S$ is a finite set.
\end{lemma}
\begin{proof}
Let $V_F\in S$, $F=\{t_n\}_{n\in\N}$ with breadth $\delta_F$, and fix an extension $u$ of $v$ to $\overline{K}$. By Proposition \ref{prop:qsn} and Corollary \ref{cor:wEphi} there are $\lambda\inZ$, $\gamma\in\Gamma_v$ depending on $F$ and $\phi(X)$ such that
\begin{equation}\label{delta=wFphi}
\delta=w_F(\phi)=\lambda \delta_F+\gamma.
\end{equation}
Since $v(\phi(t_n))$ is eventually strictly increasing, $F$ is of algebraic type and by Proposition \ref{prop:qsn} its dominating degree $\lambda$ is positive, i.e., some zero of $\phi$ is a pseudo-limit of $F$ with respect to $u$. Hence, $S$ is the union of $S_{\beta}=\{V_F\in S\mid \beta\in\limiti_F^u\}$, as $\beta$ ranges among the zeroes of $\phi$. Since $\phi$ has only finitely many zeroes, it is enough to show that each $S_{\beta}$ is finite.

Let $A_\beta$ be the set of breadths of the pseudo-convergent sequences in $S_{\beta}$; then, the cardinality of $A_\beta$ is equal to the cardinality of $S_\beta$, by Theorem \ref{teor:equiv-VE}. Let $\theta_1<\cdots<\theta_a$ be the elements of $\Gamma_v$ such that there is a critical point $\beta'$ of $\phi$ with $v(\beta-\beta')=\theta_i$; let $\theta_0=-\infty$ and $\theta_{a+1}=+\infty$. We claim that $A_{\beta}\cap(\theta_i,\theta_{i+1})$ has at most one element, for every $i\in\{0,\ldots,a\}$.

Let $V_F\in S_{\beta}$ be such that $\delta_F\in(\theta_i,\theta_{i+1})$, and let $F=\{t_n\}_{n\in\N}$. Note that for such pseudo-convergent sequences $F$, the values of $\lambda$ and $\gamma$ in (\ref{delta=wFphi}) do not depend on $F$ (explicitly, $\lambda$ is the weighted sum of critical points $\beta'$ of $\phi$ such that $v(\beta-\beta')\geq\delta_F$, which is equivalent to  $v(\beta-\beta')> \theta_{i+1}$, and $\gamma$ is defined as in Proposition \ref{prop:qsn}). In particular, by (\ref{delta=wFphi}),  $\delta_F$ is uniquely determined in $(\theta_i,\theta_{i+1})$ (recall that if $V_F\in S$ then the dominating degree is nonzero), and since we are dealing with pseudo-convergent sequences $F$ having $\beta$ as pseudo-limit, by Theorem \ref{teor:equiv-VE} $|A_\beta\cap(\theta_i,\theta_{i+1})|\leq 1$. Therefore,
\begin{equation*}
A_{\beta}\subseteq\{\theta_1\ldots,\theta_a\}\cup\bigcup_{i=0}^a(A_{\beta}\cap(\theta_i,\theta_{i+1}))
\end{equation*}
is finite. Hence, $S_{\beta}$ is finite and the claim is proved.
\end{proof}

If $V$ is a DVR, then we have already remarked at the beginning of Section \ref{sect:spaces} that $\mathcal{W}$ is a subset of $\mathcal{V}$; in particular, it is  a topological embedding. If $V$ is non-discrete, we still have an inclusion, which however is not an embedding.
\begin{prop}\label{prop:WV-continuo}
Let $V$ be a rank one non-discrete valuation domain. Let $\Psi$ be the map
\begin{equation*}
\begin{aligned}
\Psi\colon\mathcal{W} & \longrightarrow\insiemeVE\\
W_E & \longmapsto V_E
\end{aligned}
\end{equation*}
Then, $\Psi$ is continuous and injective, but it is not a topological embedding.
\end{prop}
\begin{proof}
By Theorem \ref{teor:equiv-VE}, $\Psi$ is injective. To show that $\Psi$ is continuous, it is enough to show that every $\Psi^{-1}(B^{\insiemeVE}(\phi))$ is open.

Since $V_E\subseteq W_E$, we have $\Psi^{-1}(B^{\insiemeVE}(\phi))=\{W_E\in\mathcal{W}\mid V_E\ni\phi\}\subseteq B^{\mathcal{W}}(\phi)$, and the inclusion can be strict; more precisely,
\begin{equation*}
C=B^{\mathcal{W}}(\phi)\setminus \Psi^{-1}(B^{\insiemeVE}(\phi))=\{W_E\mid \phi\in W_E\setminus V_E\}=\{W_E\mid \phi\in W_E^\ast\setminus V_E\}.
\end{equation*}
If $E=\{s_n\}_{n\inN}$ is such that $\phi\in W_E^\ast$, then $w_E(\phi)=0$; furthermore, if $\phi\notin V_E$ then $v(\phi(s_n))$ is eventually negative. Hence, for every $W_E\in C$ we must have $v(\phi(s_n))\nearrow 0$, and by Lemma \ref{lemma:finiteness} the set $C$ is finite (and possibly empty); since $\mathcal{W}$ is $T_1$ (Proposition \ref{prop:spazioW}), $C$ is closed. Hence,
\begin{equation*}
\Psi^{-1}(B^{\insiemeVE}(\phi))=B^{\mathcal{W}}(\phi)\cap (\mathcal{W}\setminus C)
\end{equation*}
is open, and so $\Psi$ is continuous.

Let $\insiemeVE_0$ be the image of $\Psi$: to show that $\Psi$ is not a topological embedding, it is enough to show that $\Phi=\Psi^{-1}:\insiemeVE_0\longrightarrow\mathcal{W}$ is not continuous. Take a pseudo-convergent sequence $E$ of algebraic type with breadth $\delta\in\Gamma_v$, and let $\zeta>\delta$. By Proposition \ref{prop:limitcostr}, if, for each $n\in\N$, $E_n$ is a pseudo-convergent sequence with limit $s_n$ and breadth $\zeta$, then $V_E$ is the limit of $V_{E_n}$ in the Zariski topology; note that both $V_E$ and the $V_{E_n}$ belong to $\insiemeVE_0$ since they have finite breadth.

Hence, if $\Phi$ were continuous then $\Phi(V_{E_n})=W_{E_n}$ would have limit $\Phi(V_E)=W_E$ in $\mathcal{W}$; since the Zariski and the constructible topologies agree on $\mathcal{W}$ (by Proposition \ref{prop:spazioW}), it would follow that $W_{E_n}$ has limit $W_E$ in $\Zar(K(X)|V)^\cons$. However, this contradicts Proposition \ref{prop:limitcostr}, since $\Zar(K(X)|V)^\cons$ is Hausdorff and $V_E\neq W_E$ by Theorem \ref{teor:rank-VE} (and the choice of $\delta$). Hence, $\Phi$ is not continuous and $\Psi$ is not a topological embedding.
\end{proof}

\subsection{Separation properties of $\insiemeVE$}\label{sect:separation}

A topological space is \emph{regular} if every point is closed and if, whenever $C$ is a closed set and $x\notin C$ then $x$ and $C$ can be separated by open sets. The space $\Zar(K(X)|V)$ is not regular under the Zariski topology, since it is not even $T_1$; on the other hand, under the constructible topology, $\mathcal{V}$ is regular, since it is a subspace of the regular space $\Zar(K(X)|V)^\cons$. In particular, by Proposition \ref{prop:costr-VE}, if the residue field of $V$ is finite then $\mathcal{V}$ is regular even if endowed with the Zariski topology (since in this case the two topologies coincide on $\mathcal{V}$). In this section, we show that the regularity of $\mathcal{V}$ under the Zariski topology holds without any additional hypothesis.


We say that two subsets $C_1,C_2$ of a topological space $X$ can be \emph{separated by open-closed sets} (\emph{open sets}, respectively) if there are disjoint open-closed (open, respectively) subsets $\Omega_1,\Omega_2$ of $X$ such that $C_i\subseteq\Omega_i$. If $C_1=\{c_1\}$ is a singleton, we also say that $c_1$ and $C_2$ can be separated by open-closed sets (open sets, respectively). 

We need a preliminary lemma.
\begin{lemma}\label{lemma:corone-disj}
Let $\gamma\in\insQ\Gamma_v$ and $s\in K$. Then, the set
\begin{equation*}
\Omega(s,\gamma)=\{V_E\in\insiemeVE\mid w_E(X-s)\leq\gamma\}
\end{equation*}
is both open and closed in $\insiemeVE$.
\end{lemma}
\begin{proof}
If $V$ is discrete, then by Theorem \ref{teor:Vdisc} $\Omega(s,\gamma)$ is homeomorphic to $\{x\in \widehat K\mid v(x-s)\leq \gamma\}$ and thus it is both open and closed since $\widehat K$ is an ultrametric space.

Suppose $V$ is not discrete, and let $\Omega=\Omega(s,\gamma)$. Let $k>0$ be an integer such that $k\gamma\in\Gamma_v$, and let $c\in K$ be such that $v(c)=k\gamma$. We claim that
\begin{equation*}
\Omega=B\left(\frac{c}{(X-s)^k}\right)
\end{equation*}
and that
\begin{equation*}
\insiemeVE\setminus\Omega=\bigcup_{\substack{d\in K\\ v(d)>\gamma}}B\left(\frac{X-s}{d}\right).
\end{equation*}
Clearly, both right hand sides are open in $\insiemeVE$.

Let $E=\{s_n\}_{n\inN}$ be a pseudo-convergent sequence. Then $v(s_n-s)$ is either eventually increasing or eventually constant, and its limit is $w_E(X-s)$ (see Remark \ref{Rem:osservazioni su wE}\ref{Rem:osservazioni su wE:X}); hence, $V_E\in\Omega$ if and only if $v(s_n-s)\leq\gamma$ for large $n$, while $V_E\notin\Omega$ if and only if $v(s_n-s)>\gamma$ for large $n$.

If $V_E\in\Omega$ then
\begin{equation*}
v\left(\frac{c}{(s_n-s)^k}\right)=v(c)-kv(s_n-s)\geq k\gamma-k\gamma=0
\end{equation*}
and so $V_E\in B\left(\frac{c}{(X-s)^k}\right)$. In the same way, if $V_E\in B\left(\frac{c}{(X-s)^k}\right)$ then $v(s_n-s)\leq\gamma$ and so $V_E\in\Omega$.

Similarly, if $V_E\notin\Omega$ then $v(s_n-s)\geq\gamma'>\gamma$ for some $\gamma'\in\Gamma_v$; if $v(d)=\gamma'$ then $V_E\in B\left(\frac{X-s}{d}\right)$ and so it is in the union. Conversely, if $V_E$ is in the union then $V_E\in B\left(\frac{X-s}{d}\right)$ for some $d$, and $v(s_n-s)\geq v(d)>\gamma$ for large $n$, so that $V_E\notin\Omega$. The claim is proved.
\end{proof}

\begin{teor}\label{teor:regolare}
$\insiemeVE$ is a regular topological space.
\end{teor}
\begin{proof}
If $V$ is a DVR, the statement follows from the fact that $\insiemeVE=\insiemeVE(\bullet, \infty)$ is an ultrametric space by \cite[Theorem 3.4]{PerTransc}. Henceforth, we assume that $V$ is not discrete.

We first note that each point of $\insiemeVE$ is closed: indeed, the closure of a point $Z$ in $\Zar(K(X)|V)$ is equal to the set of valuation domains contained in $Z$. However, two different domains $V_E$ and $V_F$ are never comparable: if they were, then $W_E=W_F$, and thus $V_E=V_F$ by Theorem \ref{teor:equiv-VE}.

Let $E=\{s_n\}_{n\in\N}\subset K$ be a pseudo-convergent sequence of breadth $\delta$, let  $\{\delta_n\}_{n\inN}$ be the gauge of $E$ and let $C\subset\insiemeVE$ be a closed set which does not contain $V_E$. Then there are rational functions $\phi_1,\ldots,\phi_k\in K(X)$ such that $V_E\in B(\phi_1,\ldots,\phi_k)$ while $B(\phi_1,\ldots,\phi_k)\cap C=\emptyset$. We let $\Lambda=\{\beta_1,\ldots,\beta_m\}\subseteq\overline{K}$ be the set of critical points of $\phi_1,\ldots,\phi_k$. Let also $u$ be an extension of $v$ to $\overline{K}$.

We want to separate $V_E$ and $C$; we need to distinguish several cases.


\caso{1} $E$ is of transcendental type.

By \cite[Theorem 31.18, p. 328]{Warner}, there is an $n$ such that no $\beta\in\Lambda$  satisfies $u(\beta-s_n)\geq\delta_n$. Hence, there is a $\gamma<\delta_n$, $\gamma\in\insQ\Gamma_v$ such that each $\beta\in\Lambda$ satisfies $u(\beta-s_n)<\gamma$. Moreover, up to considering a bigger $n\in\N$, we may also suppose that $\phi_i(s_n)\in V$ for all $i=1,\ldots,k$. Let $s=s_n$. By Theorem \ref{teor:linear}, we have $v(\phi_i(t))=v(\phi_i(s))\geq0$ for all $t$ such that $v(t-s)\geq\gamma$ and for all $i=1,\ldots,k$. 

We claim that $\Omega(s,\gamma)$ and its complement separate $C$ and $V_E$; by Lemma \ref{lemma:corone-disj}, this will imply that $C$ and $V_E$ are separated by open-closed sets.

Indeed, clearly $w_E(X-s)=\delta_{n}>\gamma$ and so $V_E\notin\Omega(s,\gamma)$. On the other hand, if $V_F\in C$ and $F=\{t_n\}_{n\inN}$, then there is an $i$ such that $v(\phi_i(t_n))$ is eventually negative. By the previous paragraph $v(t_n-s)<\gamma$ for all sufficiently large $n$; hence, $w_F(X-s)=\lim_nv(t_n-s)\leq\gamma$ and $V_F\in\Omega(s,\gamma)$. Thus, $C\subseteq\Omega(s,\gamma)$, as claimed.

\caso{2} $E$ is of algebraic type without pseudo-limits in $K$.

Let $\alpha\in\overline{K}\setminus K$ be a pseudo-limit of $E$ with respect to $u$. By Lemma \ref{lemma:kaplansky}, there is no element $t$ of $K$ such that $u(\alpha-t)\geq\delta$. By Proposition \ref{lemma:corona}, there is an annulus $\corona=\corona_u(\alpha,\tau,\delta)$ such that $\phi_i(t)\in V$ for all $i=1,\ldots,k$ and all $t\in\corona$; let $s\in\corona$ and let $\delta'=u(\alpha-s)\in\insQ\Gamma_v$. Note that $\tau<\delta'<\delta$. We claim that $\Omega(s,\tau)$ and its complement separate $C$ and $V_E$.

Indeed, we have
\begin{equation*}
w_E(X-s)=\lim_{n\to\infty}v(s_n-s)=\lim_{n\to\infty}u(s_n-\alpha+\alpha-s)=u(\alpha-s)=\delta'
\end{equation*}
since $\delta_n>\delta'$ for large $n$; hence, $V_E\not\in\Omega(s,\tau)$. On the other hand, if $F=\{t_n\}_{n\in\N}$ is a pseudo-convergent sequence such that $V_F\in C\setminus\Omega(s,\tau)$, then $\tau<v(t_n-s)$ for all $n\in\N$ sufficiently large. Therefore, for each such $n$ we have $\delta>u(t_n-\alpha)=u(t_n-s+s-\alpha)>\tau$. By our assumption this would imply that $\phi_i(t_n)\in V$ for $i=1,\ldots,k$, for all $n\geq N$, which is a contradiction since $C\cap B(\phi_1,\ldots,\phi_k)=\emptyset$. Hence, $C\subseteq \Omega(s,\tau)$, and we have proved that $C$ and $V_E$ can be separated by an open-closed set.

\caso{3} $E$ is of algebraic type and there exists a pseudo-limit $\alpha$ of $E$ in $K$.

We partition $C$ into the following three sets:
\begin{align*}
C_1= & \{V_F\in C\mid w_F(X-\alpha)<\delta\},\\
C_2= & \{V_F\in C\mid w_F(X-\alpha)>\delta\},\\
C_3= & \{V_F\in C\mid w_F(X-\alpha)=\delta\}.
\end{align*}

By Theorem \ref{teor:linear} (and Remark \ref{oss:linear}), we can find $\zeta_1,\zeta_2\in\insQ\Gamma_v$ such that $\zeta_1<\delta\leq\zeta_2$ and such that $v(\phi_i(t))=\lambda_iv(t-\alpha)+\gamma_i$ for every $t\in\corona(\alpha,\zeta_1,\zeta_2)$, for some $\lambda_i\in\Z$ and $\gamma_i\in\Gamma_v$. Since $V_E\in B(\phi_1,\ldots,\phi_k)$, by Proposition \ref{lemma:corona} we can find $\theta_1,\theta_2\in\insQ\Gamma_v$, with $\delta\in(\theta_1,\theta_2]\subseteq(\zeta_1,\zeta_2]$, such that $\phi_i(t)\in V$ for all $t\in\corona(\alpha,\theta_1,\theta_2)$  and all $i=1,\ldots,k$.

Consider $\Omega(\alpha,\theta_1)$. We have $w_E(X-\alpha)=\delta>\theta_1$, and so $V_E\notin\Omega(\alpha,\theta_1)$; on the other hand, if $V_F\in C_1$, with $F=\{t_n\}_{n\inN}$, then $v(t_n-\alpha)<\theta_1$ for all large $n$ (because $C_1\subseteq C$ has empty intersection with $B(\phi_1,\ldots,\phi_k)$) and thus also $w_F(X-\alpha)\leq\theta_1$; hence, $C_1\subseteq\Omega(\alpha,\theta_1)$. Thus, $\Omega(\alpha,\theta_1)$ and its complement are open-closed subsets separating $C_1$ and $V_E$.

Similarly, $w_E(X-\alpha)=\delta\leq\theta_2$ and thus $V_E\in\Omega(\alpha,\theta_2)$; if $V_F\in C_2$, $F=\{t_n\}_{n\inN}$, then $v(t_n-\alpha)>\theta_2$ for all large $n$ (because $C_2\subseteq C$ has empty intersection with $B(\phi_1,\ldots,\phi_k)$) and, since $v(t_n-\alpha)$ is either eventually strictly increasing or eventually constant, we have $w_F(X-\alpha)>\theta_2$, i.e., $C_2\cap\Omega(\alpha,\theta_2)=\emptyset$. Hence, $\Omega(\alpha,\theta_2)$ and its complement separate $V_E$ and $C_2$. In particular, if $C_3=\emptyset$ then $V_E$ and $C$ can be separated by open-closed sets.

Suppose $C_3\neq\emptyset$ and let $V_F\in C_3$, $F=\{t_n\}_{n\in\N}$: then $\delta\in\insQ\Gamma_v$, for otherwise $v(t_n-\alpha)$ should increase to $\delta$, and so $t_n$ would enter in any annulus $\corona(\alpha,\tau,\delta)$ and by Proposition \ref{prop:qsn} $\phi_i\in V_F$ for $i=1,\ldots,k$, against the fact that $C\cap B(\phi_1,\ldots,\phi_k)=\emptyset$. By the same argument, $v(t_n-\alpha)$ is constantly equal to $\delta$ (which therefore is in $\Gamma_v$). In particular, $\alpha$ is not a pseudo-limit of $F$ so that $\delta_F>v(t_n-\alpha)=\delta=\delta_E$.

Since $C\cap B(\phi_1,\ldots,\phi_k)=\emptyset$, for every $V_F\in C$ there is an $i\in\{1,\ldots,k\}$ such that $\phi_i(t_n)\notin V$ for all $n$ sufficiently large; for such an $i$, $w_F(\phi_i)\leq 0$ and if equality holds then $v(\phi_i(t_n))\nearrow0$, where $F=\{t_n\}_{n\in\N}$. For each $i=1,\ldots,k$, let 
\begin{equation*}
\begin{aligned}
D_i= & \{V_F\in C_3\mid w_F(\phi_i)<0\}\quad\text{and}\\
H_i= & \{V_F\in C_3\mid w_F(\phi_i)=0,~\phi_i\notin V_F\},
\end{aligned}
\end{equation*}
so that $C_3=\bigcup_{i=1,\ldots,k}(D_i\cup H_i)$.

We claim that every $D_i$ can be separated from $V_E$ by open sets: indeed, let
\begin{equation*}
\Omega_i=\bigcup_{\substack{d\in K\\ v(d)<0}}B\left(\frac{d}{\phi_i(X)}\right).
\end{equation*}
As in the proof of Lemma \ref{lemma:corone-disj}, if $V_F\in D_i$ then there is a $\kappa<0$ such that $v(\phi_i(t_n))\leq\tau$ for all large $n$ and thus, taking $d\in K$ such that $0> v(d)\geq\kappa$,
\begin{equation*}
v\left(\frac{d}{\phi_i(t_n)}\right)\geq\kappa-v(\phi_i(t_n))\geq 0
\end{equation*}
and so $V_F\in\Omega_i$. Moreover, $\Omega_i\cap B(\phi_1,\ldots,\phi_k)=\emptyset$, since otherwise there should be a $t\in K$ such that 
\begin{equation*}
\begin{cases}
v(\phi_i(t))\geq 0\\
v\left(\frac{d}{\phi_i(t)}\right)\geq 0;
\end{cases}
\end{equation*}
for some $d\in K$ such that $v(d)<0$, but the latter condition implies that $v(\phi_i(t))\leq v(d)<0$. Hence, $B(\phi_1,\ldots,\phi_k)$ and $\Omega_i$ separate $V_E$ and $D_i$.

Since for every $V_F\in H_i$, with $F=\{t_n\}_{n\in\N}$,  we have $v(\phi_i(t_n))\nearrow 0$, every $H_i$ is finite by Lemma \ref{lemma:finiteness}. Furthermore, some zero $\beta\in\overline K$ of $\phi_i$ is a pseudo-limit of $F$, with respect to some extension $u$ of $v$ to $\overline{K}$ (see the proof of Lemma \ref{lemma:finiteness}). If $n$ is sufficiently large, then $\delta_E<u(t_n-\beta)<\delta_F$. Let $\gamma\in\Gamma_v$ be such that $\delta_E<\gamma<u(t_n-\beta)$. If we let $t=t_n$, then $w_E(X-t)\leq\delta_E<\gamma<w_F(X-t)$ (see  Remark \ref{Rem:osservazioni su wE}\ref{Rem:osservazioni su wE:X}). Then $V_E\in\Omega(t,\gamma)$ and $V_F\notin\Omega(t,\gamma)$. Hence, $V_F$ can be separated from $V_E$ by the open-closed set $\Omega(t,\gamma)$, and since $H_i$ is finite it can also be separated from $V_E$ by open-closed sets.

\smallskip

To summarize, we have
\begin{equation*}
C=C_1\cup C_2\cup\bigcup_{i=1}^kD_i\cup\bigcup_{i=1}^kH_i,
\end{equation*}
and each of the sets on the right hand side can be separated from $V_E$ by open sets; since the union is finite, $C$ and $V_E$ can be separated and therefore $\insiemeVE$ is regular.
\end{proof}

As a consequence, we can show that under some conditions $\insiemeVE$ is metrizable.
\begin{cor}\label{cor:Knum-metriz}
Let $V$ be a countable valuation domain. Then, $\insiemeVE$ is metrizable.
\end{cor}
\begin{proof}
A basis for $\insiemeVE$ is $\mathcal{B}=\{B(\phi_1,\ldots,\phi_k)\mid \phi_1,\ldots\phi_k\in K(X)\}$. Since $V$ is countable, so are $K$ and $K(X)$; hence, the number of finite subsets of $K(X)$ is countable, and thus also $\mathcal{B}$ is countable. Therefore, $\insiemeVE$ is second-countable; since it is regular (Theorem \ref{teor:regolare}), it follows from Urysohn's metrization theorem \cite[e-2]{encytop} that $\insiemeVE$ is metrizable.
\end{proof}

\end{document}